\documentclass[11pt]{article}
\usepackage[top=1in, right=1in, left=1in, bottom=1in]{geometry}
\usepackage{graphicx} %
\usepackage{subfigure}
\usepackage{amsmath, amssymb, amsfonts, amsthm, mathtools}
\usepackage{empheq}
\usepackage[hidelinks]{hyperref}
\usepackage{xcolor}
\usepackage{algorithm}
\usepackage{algpseudocode}
\usepackage{enumitem}

\title{\textbf{Complexity of Minimizing Regularized Convex Quadratic Functions} \thanks{École Polytechnique Fédérale de Lausanne (EPFL), Machine Learning and Optimization Laboratory (MLO),
Switzerland. This work was supported by the Swiss State Secretariat for Education, Research and Innovation (SERI) under contract number 22.00133.}}

\author{Daniel Berg Thomsen \thanks{\texttt{\href{mailto:danielbergthomsen@gmail.com}{danielbergthomsen@gmail.com}}} 
\and 
Nikita Doikov 
\thanks{\texttt{\href{mailto:nikita.doikov@epfl.ch}{nikita.doikov@epfl.ch}}} }

\date{EPFL\\[\baselineskip]  \today}

\DeclareMathOperator{\spn}{span}

\newcommand{\Ar}{A_{\star}}
\newcommand{\Qmod}{Q_{\star}}
\newcommand{\xopt}{x_{\star}}
\newcommand{\kryl}{\mathcal{E}}

\DeclareMathOperator{\Diag}{diag}
\DeclareMathOperator*{\argmin}{arg\,min}

\newcommand{\R}{\mathbb{R}}
\newcommand{\Lr}{L_{\star}}
\newcommand{\mur}{\mu_{\star}}

\newcommand{\ub}{\mathcal{O}}
\newcommand{\lb}{\Omega}

\newtheorem{theorem}{Theorem}
\newtheorem*{theorem*}{Theorem}
\newtheorem*{remark*}{Remark}
\newtheorem{proposition}[theorem]{Proposition}%
\newtheorem{lemma}[theorem]{Lemma}
\newtheorem{corollary}[theorem]{Corollary}
\newtheorem{remark}[theorem]{Remark}%

\def\beq{\begin{equation}}
\def\eeq{\end{equation}}
\def\ba{\begin{array}}
\def\ea{\end{array}}
\def\la{\langle}
\def\ra{\rangle}

\begin{document}

\maketitle

\begin{abstract}
    In this work,
    we study the iteration complexity of gradient methods for minimizing convex quadratic functions regularized by powers of Euclidean norms. We show that, due to the uniform convexity of the objective, gradient methods have improved convergence rates. Thus, for the basic gradient descent with a novel step size, we prove a convergence rate of $\mathcal{O}(N^{-p/(p - 2)})$ for the functional residual, where $N$ is the iteration number and $p > 2$ is the power of the regularization term. We also show that this rate is tight by establishing a corresponding lower bound for one-step first-order methods. Then, for the general class of all multi-step methods, we establish that the rate of $\mathcal{O}(N^{-2p/(p-2)})$ is optimal, providing a sharp analysis of the minimization of uniformly convex regularized quadratic functions. This rate is achieved by the fast gradient method. A special case of our problem class is $p=3$, which is the minimization of cubically regularized convex quadratic functions. It naturally appears as a subproblem at each iteration of the cubic Newton method. Therefore, our theory shows that the rate of $\mathcal{O}(N^{-6})$ is optimal in this case. We also establish new lower bounds on minimizing the gradient norm within our framework.
\end{abstract}

\section{Introduction}

In this paper, we consider regularized quadratic optimization problems of the following form, for $p \geq 2$:
\begin{equation}\label{eq:prob}
    \min_{x \in \R^d}\biggl[ 
    \, 
    f(x) 
    \; := \;
    \frac{1}{2} x^\top A x - b^\top x + \frac{s}{p} \|x\|^p 
    \, \biggr],
\end{equation}
where $A = A^{\top} \succeq 0$ is 
a symmetric and positive semidefinite matrix, $b \in \R^d$, 
and $s > 0$ is a regularization parameter.
The norm $\| \cdot \|$ is the standard Euclidean.
We denote by $L \geq \mu \geq 0$ the largest and the smallest eigenvalues of the matrix: $L = \lambda_{\max}(A)$, $\mu = \lambda_{\min}(A)$,
which are important parameters of our problem class.

The most popular example of problem~\eqref{eq:prob} is when $p=2$.
In this case, the problem is strongly convex,
and a closed-form expression for the solution is given by
\begin{equation}\label{eq:solution_characterization}
    \xopt = \big(A + s I \big)^{-1} b.
\end{equation}
In general, for arbitrary $p \geq 2$,
the problem is \emph{uniformly convex} \cite{nesterov_accelerating_2008}
and always has the unique solution $\xopt$.
However, for $p > 2$, we do not have an explicit expression
for $\xopt$ using only a matrix inversion operation.
In this case, the solution to~\eqref{eq:prob}
can be found by solving the following univariate nonlinear equation:
\beq \label{StatCondition}
\ba{rcl}
A \xopt + s\|\xopt\|^{p - 2} \xopt & = & b.
\ea
\eeq
It can be solved efficiently
by employing matrix factorization techniques from computational linear algebra
\cite{conn_trust_2000,nesterov_cubic_2006}.

In a large-scale setting,
an alternative approach consists
of applying first-order gradient methods
to \eqref{eq:prob}. These methods
typically have a low per-iteration complexity,
with the main operation being a computation
of a matrix-vector product.
However, they might require 
a significant number of iterations
to converge. Therefore, it is crucial
to study worst-case bounds for the convergence rate
of these methods.

The first \textit{lower bounds} on the number of oracle calls 
required to solve an optimization problem,
and the notion of \textit{problem class}, 
were introduced in~\cite{nemirovski_problem_1983},
initiating the development of the complexity theory of optimization.
Such results are important because they give information about the fundamental difficulty of a given optimization problem, and, in the presence of an upper bound, 
show how far it is from the optimal one.
Complexity-theoretic results allow us to characterize how much computational resources are required for minimizing functions from a particular
problem class.

For problem~\eqref{eq:prob},
the best-known convergence rate 
in terms of the functional residual
was 
established~\cite{roulet2017sharpness,nesterov_inexact_2022}
for the composite version of the Fast Gradient Method~\cite{nesterov_method_1983},
which is, for $p > 2$ and $s > 0$:
\beq \label{FGM_rate}
\ba{c}
\mathcal{O}\bigl( 1 / N^{\frac{2p}{p - 2}}\bigr),
\ea
\eeq
where $N$ is the number of iterations
of the method.
Thus, we see that first-order methods
can exploit the uniform convexity of the problem
and achieve faster rates of convergence
than $\mathcal{O}(1 / N^2)$, which is for minimizing a convex smooth objective ($s = 0$).
It is widely known that the Fast Gradient Method
is \textit{optimal} for the standard classes
of convex and strongly convex functions
with Lipschitz continuous gradient~\cite{nesterov_lectures_2018},
but it has remained an open question whether
convergence rate~\eqref{FGM_rate} is the best possible
for first-order methods
minimizing regularized quadratic functions of the form~\eqref{eq:prob}.

A line of work has focused on justifying 
complexity lower bounds for the function classes with non-standard
smoothness or convexity assumptions \cite{nemirovskii1985optimal,guzman_lower_2015, agarwal_lower_2018, arjevani_oracle_2019, diakonikolas2024complementary}.  Lower bounds
matching \eqref{FGM_rate} were recently proven for the classes of uniformly convex functions 
with H\"older continuous gradient in \cite{doikov_lower_2022}. 
However, the construction of this lower bound
employs a smoothing technique,
while the target resisting function is not a regularized quadratic function
of the form~\eqref{eq:prob}. 
It has therefore been unknown
whether rate~\eqref{FGM_rate} is tight
for quadratic objectives.

\paragraph{Contributions}
The contributions of this paper are summarized as follows:
\begin{enumerate}[label=\textbf{(\roman*)}]
\item First, we prove a convergence rate of $\ub(N^{-p/(p-2)})$ for the simple version of the basic gradient descent solving~\eqref{eq:prob} using a novel step size rule that we introduce in Section~\ref{sec:gradient_method}. 
Previously, the known approach for achieving this convergence rate involved using the \textit{composite gradient method}~\cite{nesterov_implementable_2021}, which produces an iteration sequence that is different from ours. Therefore, our approach complements previous methods by demonstrating the efficiency of the basic gradient steps with a properly chosen step sizes.
We also present an adaptive version of our method,
which automatically adjusts to a local smoothness and does not require any unknown parameters of the objective, such as the Lipschitz constant~$L$ or the size of the solution~$\|\xopt\|$. Notably, our adaptive scheme achieves the same convergence rate of~$\ub(N^{-p/(p-2)})$, where $N$ is the number of iterations (matrix-vector products).
Thus, compared to the standard rate $\ub(N^{-1})$ 
for convex functions, we observe that the rate of the basic gradient method accelerates due to the uniform convexity of the regularizer in~\eqref{eq:prob}.

\item 
We establish that the rate of~$\ub(N^{-p/(p-2)})$ for the basic gradient method
is \textit{optimal}, by proving the matching lower bound 
for one-step gradient methods on our problem class in Section~\ref{sec:one_step_lower_bounds}.
Our framework covers a large variety of one-step gradient methods,
including both the classical and the composite gradient methods
as applied to objective~\eqref{eq:prob}.

\item 
While the rate~$\ub(N^{-p/(p-2)})$
is optimal for one-step gradient methods,
it is known that it can be \textit{accelerated}
achieving~\eqref{FGM_rate} by taking into account the history of past iterations~\cite{nesterov_implementable_2021}.
We address the open question of establishing a matching lower bound
of $\lb(1 / N^{\frac{2p}{p - 2}})$, for the general class of \textit{all} first-order methods 
by proving the following theorem:

\begin{theorem*}
    Let $p > 2$, $s > 0$, and $L > 0$. For any first-order method 
    running for $N \leq \frac{d - 1}{2}$ iterations,
    there is a function $f$ of the form in \eqref{eq:prob}
    with $\lambda_{\max}(A) = L$ such that
    \begin{equation}
    \ba{rcl}
        f(x_N) - f_\star & \geq & \lb \Big( \frac{L}{s^{2/p} N^2}\Big)^{\frac{p}{p-2}},
    \ea
    \end{equation}
    where $\lb(\cdot)$ hides an absolute numerical constant.
\end{theorem*}
The full theorem statement with proof can be found in Section~\ref{sec:implications}. 
More general results that apply broadly to our class of problems are proven in Section~\ref{sec:lower_bounds}. 
Our result is continuous in $2 \leq p < +\infty$
and thus it also recovers the known complexity bounds for the class of smooth strongly convex functions ($p = 2$).
\item We demonstrate experiments confirming our theory in Section~\ref{sec:experiments}, and show that 
we can relax the theoretical construction of the lower bound, while still observing the worst-case performance
of the first-order methods.
\end{enumerate}

\paragraph{Related work}
Oracle lower bounds for the class
of smooth convex functions frequently employ quadratic functions
in their constructions. 
A classic example is a lower bound 
of $\lb\big(N^{-2}\big)$ for gradient methods
minimizing smooth convex functions
that is built by using a particular quadratic function with a tridiagonal matrix, and assuming that the iterates
of the method remain within the span of the observed gradients~\cite{nesterov_lectures_2018}.
A similar construction can also be used to prove the corresponding lower bound for smooth strongly convex functions.
These rates are achieved by the Fast Gradient Method \cite{nesterov_lectures_2018, nesterov_method_1983}. 

It is possible to avoid the assumption
that the method produces iterates from the span of previous queries, using the idea of a \textit{resisting oracle}
and constructing the worst-case quadratic function
adversarially~\cite{arkadi_nemirovski_information-based_1994}. 
Therefore, the same lower complexity bounds
hold for even large families of first-order algorithms,
including the composite gradient methods~\cite{nesterov_lectures_2018} and quasi-Newton methods~\cite{nocedal2006conjugate}.
The framework developed in the lecture notes~\cite{arkadi_nemirovski_information-based_1994} has been adapted to analyze regularized quadratic functions of the form given in~\eqref{eq:prob} in an unpublished master's thesis~\cite{master_thesis}, and this paper contains a shortened version of the proof.

An important example of a problem of form (\ref{eq:prob}) is the subproblem of the cubically regularized Newton method~\cite{griewank_modification_1981,nesterov_cubic_2006,cartis_adaptive_2011-1}. 
For minimizing a twice-differentiable function $g$, the next iteration of the cubic Newton is chosen according to the rule
\begin{equation} \label{CubicNewton}
\ba{rcl}
    x_{k+1} & \coloneq & 
    x_k + \argmin\limits_{z \in \R^d} \Big\{ \,
    \frac{1}{2} z^\top \nabla^2 g(x_k) z + \nabla g(x_k)^\top z + \frac{s}{3} \|z\|^3 
    \, \Big\},
\ea
\end{equation}
where $s$ is the Lipschitz constant of $\nabla^2g(x_k)$. 
Thus, when $\nabla^2 g(x_k) \succeq 0$, the subproblem in~\eqref{CubicNewton} exactly corresponds
to our class of functions~\eqref{eq:prob} for $p = 3$.

Another interesting case is $2 \leq p \leq 3$, which is related to subproblems
of regularized Newton methods that have been studied 
in \cite{grapiglia_regularized_2017, grapiglia_accelerated_2019}, and
the case $p = 4$, which is important
in the context of high-order tensor methods~\cite{nesterov_implementable_2021}.

In the previous work~\cite{carmon_analysis_2018}, the authors also considered cubically regularized quadratic functions, allowing $\mu := \lambda_{\min}(A) < 0$, \emph{i.e.,} non-convex problems. 
Note that due to the regularization term, for $p > 2$, the non-convex problems with $\mu < 0$ are still tractable
and attain the global minima. Therefore we can consider the complexity of finding it by the first-order methods. Lower bounds corresponding to both linear and sublinear rates were derived in \cite{carmon_analysis_2018},
by finding the correspondence between the global minimum of the original problem and to the minimum of an associated trust-region problem. 
However, in our work, we assume that our function is always convex ($\mu \geq 0$),
and therefore we are able to derive tight rates, which are significantly better than
the corresponding ones for the general non-convex case.

A closely related to the regularized Newton subproblem is the trust-region subproblem \cite{conn_trust_2000, erway_subspace_2010, gould_solving_1999, gould_solving_2010, cartis_adaptive_2011-1, hazan_linear-time_2016}. It corresponds to a constrained quadratic minimization over a given Euclidean ball.
Correspondingly, the optimal rate achieved by first-order methods
solving the convex trust-region subproblem is $\mathcal{O}(N^{-2})$,
which can be seen as a limiting case of our theory when $p \to +\infty$.

Another instance in which we see problems of the form described in (\ref{eq:prob}) is the case of training the $\ell_2$-regularized linear regression problem
\begin{equation*}
    \min_{w \in \R^d} \|Xw - y\|^2 + \lambda \|w\|^p,
\end{equation*}
where $X \in \R^{n \times d}$ is a given data matrix and $y \in \R^n$ is a given set of observations. This includes higher degrees of regularization than \emph{ridge regression} \cite{hastie_linear_2009}. The state-of-the-art methods are sketching algorithms \cite{chowdhury_iterative_2018, kacham_sketching_2022}, and there are probabilistic computational lower bounds for this problem \cite{kacham_sketching_2022}.

Related to our work, a recent trend in optimization focuses on different smoothness conditions, including the notions of \textit{relatively smooth}~\cite{van2017forward,bauschke2017descent,lu2018relatively} and \textit{$(L_0, L_1)$-smooth} functions (see~\cite{zhang2019gradient,koloskova2023revisiting,gorbunov2024methods,vankov2024optimizing} and references therein),
which aim extend beyond the standard notion of first-order smoothness. In contrast, our work studies a particular class of convex regularized quadratic functions, given in explicit form, and establishes tight lower and upper complexity bounds specific to our formulation. 

\paragraph{Notation}
We denote by $\|v\|$ the standard Euclidean norm of the vector $v$. For a matrix $B$,
we use the notation $\|B\|$ for the spectral norm.
We say that a matrix $U$ is $b$-invariant, where $b$ is a given vector, if  $Ub = b$.
We denote the standard simplex by 
$\Delta_d := \{ x \in \R^d_{\geq 0} \, : \, e^{\top} x = 1 \}$,
where $e$ is the vector of all ones. We denote the standard orthonormal
basis by $e_1, \ldots, e_d$, and by $I$ the identity
matrix. $\mathcal{P}_n$ refers to the space of polynomials
with degree~$\leq n$.

\section{Convergence of the Basic Gradient Method}
\label{sec:gradient_method}

In this section, we prove a convergence guarantee for the basic gradient method on problem~\eqref{eq:prob}.
For each iteration $k \geq 0$,
we set:
\beq \label{eq:grad_method}
\boxed{\ba{rcl}
x_{k+1} & \coloneqq & x_k - \eta_k \nabla f(x_k),
\ea}
\eeq
where $\eta_k > 0$ is a certain step size.
Without loss of generality we can assume that $x_0 := 0 \in \R^d$.
Note that our objective function $f$ has no globally Lipschitz-continuous gradient; therefore, we cannot immediately apply
the standard convergence theory of the gradient method~\cite{nesterov_lectures_2018}.
However, as we will show, we can still prove fast global convergence, by taking the
properties of our regularizer into account.

\subsection{Main Inequalities}

In our analysis, we use the standard notion of \textit{uniform convexity}~\cite{nesterov_accelerating_2008}.
Thus, we say that a differentiable function $f$ is uniformly convex of degree $p \geq 2$ if there exists a constant $\sigma_p > 0$, such that
\begin{equation}\label{eq:uniform_ineq}
\ba{rcl}
    f(y) & \geq & 
    f(x) + \nabla f(x)^\top (y-x) + \frac{\sigma_p}{p} \|y-x\|^p,
    \qquad x, y \in \R^d.
\ea
\end{equation}
Note that the power of the Euclidean norm, 
$g(x) \coloneq \frac{s}{p} \|x\|^p$, $p \geq 2$,
is uniformly convex with constant 
(see, e.g., Lemma~2.5 in \cite{doikov_minimizing_2021}):
\beq \label{eq:sigma_value}
\ba{rcl}
\sigma_p & = & s2^{2-p}.
\ea
\eeq
As a consequence, 
our regularized quadratic objective~\eqref{eq:prob}
is uniformly convex with the same parameter
\eqref{eq:sigma_value}.
Substituting $x := \xopt$ into
\eqref{eq:uniform_ineq},
we obtain the following useful bound
on the functional residual,
\beq \label{eq:func_residual}
\ba{rcl}
f(y) - f_{\star} & \geq &
\frac{\sigma_p}{p}\|y - \xopt\|^p, \qquad y \in \R^d.
\ea
\eeq
A further useful consequence follows from minimizing both sides of~\eqref{eq:uniform_ineq} with respect to $y$ (see (5) of \cite{doikov_minimizing_2021}):
\beq \label{eq:grad_ineq}
\ba{rcl}
    f(x) - f_\star 
    &\leq& 
    \frac{p-1}{p} \Big(\frac{1}{\sigma_p}\Big)^{\frac{1}{p-1}}\|\nabla f(x)\|^{\frac{p}{p-1}},
    \qquad x \in \R^d.
\ea
\eeq

Let us start with the following simple condition
which ensures that all iterates of the gradient method
will remain bounded, when the step size is small.

\begin{lemma} \label{LemmaStepSize}
For $k \geq 0$, let it hold that
\beq \label{StepsizeSufficient}
\ba{rcl}
\eta_k & \leq & 
\frac{2\la \nabla f(x_k), x_k - x_{\star}\ra}{\|\nabla f(x_k) \|^2}.
\ea
\eeq
Then
\beq \label{DistancesBound}
\ba{rcl}
\| x_{k + 1} - x_{\star} \| & \leq & \| x_k - x_{\star} \|.
\ea
\eeq
\end{lemma}
\begin{proof}
Indeed, 
$$
\ba{rcl}
\frac{1}{2}\|x_{k + 1} - x_{\star}\|^2
 & = & \frac{1}{2} \|x_k - x_{\star}\|^2
 - \eta_k \Bigl[  
 \la \nabla f(x_k), x_{k} - x_{\star} \ra
 - \frac{\eta_k}{2}\| \nabla f(x_k) \|^2
 \Bigr]
 \;\; \overset{\eqref{StepsizeSufficient}}{\leq} \;\;
 \frac{1}{2} \|x_k - x_{\star}\|^2.
\ea
$$
\end{proof}
Therefore, with a sufficiently small step size $\eta_k$ satisfying~\eqref{StepsizeSufficient}
at all iterations, we ensure that all points belong to the following Euclidean ball:
\beq \label{BallBound}
\ba{rcl}
x_k & \in & B_{\star}
\;\; := \;\; \{ y \; : \; \|y - x_{\star}\| \leq \|x_{\star} \|  \}, \qquad \forall k \geq 0.
\ea
\eeq
At the same time, we have, for any $y \in \R^d$:
\beq \label{HessBound}
\ba{rcl}
\nabla^2 f(y) & = &
A + s \|y\|^{p - 2} I + s(p - 2) \|y\|^{p - 4} yy^{\top}
\;\; \preceq \;\;
(L + s(p - 1)\|y\|^{p - 2}) I,
\ea
\eeq
where $L \geq 0$ is the largest eigenvalue of matrix $A$.
Consequently, for all points from $B_{\star}$, we can bound the Hessian of our objective using a single Lipschitz constant and thus obtain a bound on the progress of one step of the method.
We denote this Lipschitz constant by
\beq \label{MDef}
\boxed{
\ba{rcl}
M_{\star} & := & L + s(p - 1) 2^{p - 2} \|x_{\star}\|^{p - 2}.
\ea
}
\eeq

\begin{lemma} Let $x_k \in B_{\star}$. Let the step size $\eta_k$ satisfy~\eqref{StepsizeSufficient}
and $\eta_k \leq \frac{1}{M_{\star}}$.
Then,
\beq \label{OneStepProgress}
\ba{rcl}
f(x_k) - f(x_{k + 1}) & \geq & \frac{\eta_k}{2} \| \nabla f(x_k) \|^2.
\ea
\eeq
\end{lemma}
\begin{proof}
Since both $x_k \in B_{\star}$ and $x_{k + 1} \in B_{\star}$, we have
$$
\ba{cl}
&f(x_{k + 1}) - f(x_k) - \nabla f(x_k)^{\top} (x_{k + 1} - x_k) \\
\\
& \;\; = \;\; \int\limits_{0}^1 (1 - \tau) 
(x_{k + 1} - x_k)^{\top}\nabla^2 f(x_k + \tau(x_{k + 1} - x_k)) (x_{k + 1} - x_k) d\tau
\overset{\eqref{HessBound}}{\leq}
\frac{M_{\star}}{2}\|x_{k + 1} - x_k\|^2.
\ea
$$
Substituting the gradient step~\eqref{eq:grad_method} 
and using the fact that $\eta_k \leq \frac{1}{M_{\star}}$
completes the proof.
\end{proof}

Up to now, we may expect similar convergence
properties for our method as in the standard smooth case.
The main question which remains is how to satisfy condition~\eqref{StepsizeSufficient}
on the step size. Note that the right hand side of~\eqref{StepsizeSufficient}
resembles
the classical Polyak step size commonly used in both smooth and non-smooth convex optimization~\cite{polyak1987introduction,nesterov2024primal}.

From~\eqref{OneStepProgress} we observe that the value of $\eta_k$ is responsible 
for the progress of our method. Thus, to establish the convergence
of the method we have to justify that the right hand side in~\eqref{StepsizeSufficient} 
is separated from zero. 

We prove the following useful lemma,
which generalizes upon the classical first-order smoothness and takes into
account the growth of the regularizer of degree $p \geq 2$.

\begin{lemma}
    For any $x \in \R^d$, we have
    \beq \label{GradResBound}
    \ba{rcl}
    \la \nabla f(x), x - x_{\star} \ra 
    & \geq & f(x) - f_{\star}
    + \frac{\| \nabla f(x) \|^2}{2} \min\Bigl\{  \frac{1}{M_{\star}},
      \bigl[ \frac{p}{s2^{p - 2} \| \nabla f(x) \|^{p - 2}}  \bigr]^{\frac{1}{p - 1}}  \Bigr\}.
    \ea
    \eeq
\end{lemma}
\begin{proof}
Indeed, for any $h \in \R^d$, it holds that
\beq \label{GradResProof}
\ba{rcl}
\la \nabla f(x), x - x_{\star} \ra 
& = &
\la \nabla f(x), x - x_{\star} - h \ra + \la \nabla f(x), h \ra \\
\\
& \geq & 
f(x) - f(x_{\star} + h) + \la \nabla f(x), h \ra \\
\\
& = &
f(x) - f_{\star} + \la \nabla f(x), h \ra 
- (f(x_{\star} + h) - f_{\star}).
\ea
\eeq
Estimating the second derivative,
we get the following bound for the last term in~\eqref{GradResProof}:
$$
\ba{rcl}
f(x_{\star} + h) - f_{\star} & = & 
\int\limits_0^1 (1 - \tau) \la \nabla^2 f(x_{\star} + \tau h)h, h \ra d\tau \\[10pt]
& \overset{\eqref{HessBound}}{\leq} &
\frac{L}{2}\|h\|^2 + 
s(p - 1)
\int\limits_0^1 (1 - \tau) \|x_{\star} + \tau h\|^{p - 2} \|h\|^2 d\tau \\[10pt]
& \leq &
\frac{L}{2}\|h\|^2 + 
s(p - 1)
\int\limits_0^1 (1 - \tau)\Bigl[ 
2^{p - 3} \|x_{\star}\|^{p - 2} \|h\|^2 + 2^{p - 3}\tau^{p - 2} \|h\|^{p} \Bigr] d\tau \\[10pt]
& = &
\|h\|^2 \cdot \Bigl[ 
\frac{L}{2} + s(p - 1)2^{p - 4}\|x_{\star}\|^{p - 2} + \frac{s2^{p - 3}}{p} \|h\|^{p - 2}
\Bigr].
\ea
$$
Substituting this bound into~\eqref{GradResProof} and using $h := \xi \nabla f(x)$
with 
$$
\ba{rcl}
\xi & := & \min\Bigl\{  \frac{1}{M_{\star}}, 
\bigl[ \frac{p}{s 2^{p - 2} \| \nabla f(x) \|^{p - 2}} \bigr]^{\frac{1}{p - 1}} \Bigr\}
\ea
$$
completes the proof.
\end{proof}

\begin{remark}
    From~\eqref{GradResBound}, we obtain the following important inequality,
    which we use to define the local step size for our method, for any $x \in \R^d$:
    \beq \label{GradResBound2}
    \ba{rcl}
    \frac{2 \la \nabla f(x), x - x_{\star} \ra}{ \| \nabla f(x) \|^2 }
    & \geq &
    \eta_{\star}(x) \;\; := \;\;
    \min\Bigl\{ \frac{1}{M_{\star}},
      \bigl[ \frac{p}{s2^{p - 2} \| \nabla f(x) \|^{p - 2}}  \bigr]^{\frac{1}{p - 1}}  \Bigr\},
    \ea
    \eeq
    where $M_{\star}$ is defined by~\eqref{MDef}.
\end{remark}

Finally, we also provide a bound on the functional residual at the origin for this class of functions.
\begin{lemma}\label{lem:initial_residual}
Let $p > 2$. Then we have the following bound on the functional residual at the origin:
\beq \label{FuncBoundOrigin}
\ba{rcl}
f(0) - f(\xopt) & \leq & 
2^{-3} \cdot ( (p - 1) s)^{-\frac{2}{p-2}} \cdot
M_\star^{\frac{p}{p-2}}.
\ea
\eeq
\end{lemma}
\begin{proof}
    First, we note that
    \beq
    \ba{rcl}
    f(0) - f(\xopt) & = & b^\top \xopt - \frac{1}{2} \xopt^\top A \xopt - \frac{s}{p} \|\xopt\|^p 
    \;\; \stackrel{\eqref{StatCondition}}{=} \;\;
    \frac{1}{2} \xopt^\top A \xopt + s\Big(\frac{p-1}{p}\Big) \|\xopt\|^p\\
    \\
    & \leq & \frac{\|\xopt\|^2}{2} \Big[L + 2s \Big(\frac{p-1}{p}\Big)\|\xopt\|^{p-2}\Big]
    \;\; \leq \;\; \frac{\|\xopt\|^2}{2} \cdot M_\star,
    \ea
    \eeq
where the last two inequalities follow from the definitions of $L$ and $M_\star$.
It remains to notice that
\beq \label{MOptBound}
\ba{rcl}
M_{\star} & := & L + s(p - 1)2^{p - 2}\|\xopt\|^{p - 2}
\;\; \geq \;\; s(p - 1)2^{p - 2}\|\xopt\|^{p - 2},
\ea
\eeq
and therefore,
$$
\ba{rcl}
\frac{\| \xopt \|^2}{2} & \overset{\eqref{MOptBound}}{\leq} &
\frac{1}{2} \Bigl(  \frac{M_{\star}}{s(p - 1)2^{p - 2}} \Bigr)^{\frac{2}{p - 2}}
\ea
$$
which proves the desired bound.
\end{proof}

\subsection{Convergence Rates}

From the previous results, we immediately obtain a step size rule to choose
$\eta_k$, for every $k \geq 0$:
\beq \label{ConstantEtaK}
\boxed{
\ba{rcl}
\eta_k & := & \eta_{\star}(x_k),
\ea
}
\eeq
where $\eta_{\star}(\cdot)$ is defined by~\eqref{GradResBound2},
and
which guarantees progress condition~\eqref{OneStepProgress}
for every step of our method.
The main drawback is that it requires knowledge of the constant $M_{\star}$,
which includes the Lipschitz constant $L$ and the bound on the solution $\|x_{\star}\|$.
Instead, we can use a simple backtracking line-search procedure (see Algorithm~\ref{alg:line_search}) that sets
\beq \label{BacktrackingEtaK}
\boxed{
\ba{rcl}
\eta_k & := &
\min\Bigl\{ \frac{1}{M_k}, \bigl[ \frac{p}{s2^{p - 2} \| \nabla f(x_k) \|^{p - 2}} \bigr]^{\frac{1}{p - 1}}   \Bigr\},
\ea
}
\eeq
for some local estimates $M_k$ of the parameter $M_{\star}$. We only assume that
these coefficients are bounded:
\beq \label{BoundedMk}
\ba{rcl}
M_k & \leq & \max\{ 2M_{\star}, M_0 \},
\ea
\eeq
and that the progress inequality~\eqref{OneStepProgress} is satisfied
for our choice of $M_k$. These assumptions clearly hold both for the theoretical
step size rule~\eqref{ConstantEtaK} and for the iterations of our adaptive Algorithm~\ref{alg:line_search}.

Now, we are ready to establish a convergence rate for the method.

\begin{theorem} \label{TheoremGMRate}
Let $\{x_k\}_{k \geq 0}$ be iterations of the gradient method~\eqref{eq:grad_method}
starting from $x_0 = 0 \in \R^d$,
as applied to problem~\eqref{eq:prob}
with $p > 2$.
Let step sizes $\eta_k$ be chosen according to~\eqref{BacktrackingEtaK},
where $M_k$ is bounded~\eqref{BoundedMk}, and assume that progress inequality~\eqref{OneStepProgress} holds. 

Then the functional residual satisfies the following bound:
\beq \label{GMRate}
\ba{rcl}
f(x_k) - f_{\star} & \leq & 
F_0 \cdot \biggl[ 
\frac{8\max\{ 2M_{\star}, M_0 \}}{(4s)^{2/p}} \cdot (p-1)^{\frac{2(p-1)}{p}}
\cdot
\Bigl((pF_0)^{\frac{p-2}{p}} (p-2) k  + 1
\Bigr)^{-1}
\biggr]^{\frac{p}{p - 2}},
\ea
\eeq
where $F_0 := f(0) - f_{\star}$.
\end{theorem}
\begin{proof}
We denote $F_k := f(x_k) - f_{\star}$. 
By uniform convexity~\eqref{eq:grad_ineq}
with constant~\eqref{eq:sigma_value},
$\sigma_p = s 2^{2 - p}$ of our objective, we have the following inequality:
\beq \label{GradBoundUC}
\ba{rcl}
\| \nabla f(x_k) \|^{\frac{p}{p - 1}}
& \geq &
\bigl[ \frac{p}{p - 1}  \bigr] \cdot \bigl( s2^{2 - p} \bigr)^{\frac{1}{p - 1}} \cdot F_k.
\ea
\eeq
Therefore,
for every iteration $k \geq 0$ of the gradient method,
we have
\beq \label{FuncProgress}
\ba{rcl}
F_k - F_{k + 1} & \overset{\eqref{OneStepProgress}}{\geq} &
\frac{\eta_k}{2}\| \nabla f(x_k) \|^2
\;\; \overset{\eqref{BacktrackingEtaK}}{=} \;\;
\frac{1}{2}
\min\Bigl\{  
\frac{\| \nabla f(x_k) \|^2}{M_k},
\bigl(  \frac{p}{s2^{p - 2}} \bigr)^{\frac{1}{p - 1}}
\| \nabla f(x_k) \|^{\frac{p}{p - 1}}
\Bigr\} \\
\\
& \overset{\eqref{BoundedMk}, \eqref{GradBoundUC}}{\geq} &
\min\Bigl\{  c_2 F_k^{\frac{2(p - 1)}{p}}, \;\;c_1 F_k  \Bigr\},
\ea
\eeq
where $c_2 := \frac{1}{2} \cdot \bigl[ \frac{p}{p - 1} \bigr]^{\frac{2(p-1)}{p}}
\cdot \frac{(s 2^{2 - p})^{2 / p}}{\max\{ 2M_{\star}, M_0\}}
$
and $c_1 := \frac{1}{2} \cdot \bigl[\frac{p}{p - 1} \bigr] \cdot 
[\, p 2^{2(2 - p)} \,]^{\frac{1}{p - 1}}$.
We show that that $c_1 F_k \geq c_2 F_k^{\frac{2(p - 1)}{p}}$ always,
which is equivalent to check that
\beq \label{C1C2Bound}
\ba{rcl}
F_k & \leq & 
\Bigl[ \frac{c_1}{c_2} \Bigr]^{\frac{p}{p - 2}}
\;\; = \;\;
\biggl[ 
\frac{\max\{ 2M_{\star}, M_0 \} \cdot p^{\frac{1}{ p - 1}} }{ 
[ \frac{p}{p - 1} ]^{\frac{p - 2}{p}} \cdot s^{\frac{2}{p}}}
\cdot 2^{\frac{2(2 - p)}{p - 1} - \frac{2(2 - p)}{p}}
\biggr]^{\frac{p}{p - 2}} \\
\\
& = &
\frac{(p - 1) \cdot \max\{2M_{\star}, M_0 \}^{\frac{p}{p - 2}}}{
2^{\frac{2}{p - 1}} \cdot s^{\frac{2}{p - 2}}} 
\cdot p^{ \frac{p}{(p - 1)(p - 2)} - 1}.
\ea
\eeq
In view of monotonicity of the sequence $\{ F_k \}_{k \geq 0}$
and applying Lemma~\ref{lem:initial_residual}, we have that
$$
\ba{rcl}
F_k & \leq & F_0 \;\; \overset{\eqref{FuncBoundOrigin}}{\leq} \;\;
\frac{M_{\star}^{\frac{p}{p - 2}}}{
2^3 \cdot s^{\frac{2}{p - 2}} \cdot (p - 1)^{\frac{2}{p - 2}}.
}
\ea
$$
Therefore, to establish~\eqref{C1C2Bound}, it is sufficient to verify
the following inequality:
$$
\ba{rcl}
2^{3 + \frac{p}{p - 2} - \frac{2}{p - 1}}
\cdot 
(p - 1)^{\frac{p}{p - 2}}
\cdot 
p^{\frac{p}{(p - 1)(p - 2)} - 1}
& \geq & 1,
\ea
$$
which holds for any $p > 2$.

Hence, applying bound~\eqref{C1C2Bound} in~\eqref{FuncProgress}, we obtain the following inequality characterizing the per-iteration progress:
\beq \label{ProgressGM}
\ba{rcl}
F_k - F_{k + 1} & \geq & c_2 F_k^{\alpha}, \qquad 
\alpha \;\; := \;\; \frac{2(p - 1)}{p} \;\; \geq \;\; 1.
\ea
\eeq
By properties of the power function, for an arbitrary $\gamma > 0$ and $a > b > 0$
we have
\beq \label{PowerFuncProp}
\ba{rcl}
a - b & \geq & \frac{a^{\gamma} - b^{\gamma}}{\gamma a^{\gamma - 1}}.
\ea
\eeq
Applying this inequality with $\gamma := \frac{1}{\alpha - 1}$, we get
$$
\ba{rcl}
\frac{1}{F_{k + 1}^{\alpha - 1}} - \frac{1}{F_k^{\alpha - 1}} & = & 
\frac{ F_k^{\alpha - 1} - F_{k + 1}^{\alpha - 1} }{F_k^{\alpha - 1} F_{k + 1}^{\alpha - 1}}
\; \overset{\eqref{PowerFuncProp}}{\geq} \;
(\alpha - 1)\frac{F_k - F_{k + 1}}{F_k F_{k + 1}^{\alpha - 1}}
\; \overset{\eqref{ProgressGM}}{\geq} \; c_2 (\alpha - 1) \frac{F_k^{\alpha - 1}}{F_{k + 1}^{\alpha - 1}}
\; \geq \; c_2(\alpha - 1).
\ea
$$
Therefore, telescoping this bound for the first $k$ iterations, we have
$
F_k \leq \Bigl(\frac{F_0^{\alpha - 1}}{F_0^{\alpha - 1} c_2 (\alpha - 1) k + 1}\Bigr)^{\frac{1}{\alpha - 1}}$.
Substituting the values of constants gives us~\eqref{GMRate}, which completes the proof. 
\end{proof}

\begin{corollary}
According to Theorem~\ref{TheoremGMRate},
for any given $\varepsilon > 0$,
the gradient method needs to perform
$$
\ba{rcl}
K & = & \mathcal{O}\Bigl[
\Bigl(
\frac{L + s \|\xopt\|^{p-2}}{s^{2/p}\varepsilon}
\Bigr)^{\frac{p}{p-2}}
\Bigr]
\ea
$$
iterations to reach $f(x_K) - f_{\star} \leq \varepsilon$.
\end{corollary}

\subsection{Adaptive Step Size}

In practice, we are unlikely to know the value of $\|\xopt\|$ beforehand,
as well as the exact value of the leading eigenvalue $L$ of matrix $A$.
While the result from the previous section proves the existence of a step size~\eqref{ConstantEtaK} that will converge on our class of functions, and such a step size can be found using \emph{e.g.} a grid-search procedure, in practice we would prefer a method without such a tuning-sensitive parameter. 
In what follows we propose a simple backtracking line-search procedure that adjusts the step size automatically for our problem.

It can be seen as a modification of the standard Armijo-type backtracking rule~\cite{nesterov_lectures_2018}, which is more suitable
for our problem class. At every iteration, we keep a local estimate $M_k$ 
of the global unknown parameter $M_{\star}$, defined by~\eqref{MDef}.
We use $M_k$ to define the step size $\eta_k$ according to our theory.
Then, we increase $M_k$ until the sufficient decrease condition
is not satisfied. We also try to decrease $M_k$ at every iteration,
to exploit the existence of a potentially smaller local Lipschitz constant. Our method can be formulated as follows in algorithmic form.

\begin{algorithm}[h!]
\caption{Adaptive Gradient Method}\label{alg:line_search}
\begin{algorithmic}
\State{\bfseries initialization:} $x_0 \in \R^d, M_0 > 0$
\For{$k \gets 0, 1, \ldots $}
\State $M_+ \gets \frac{1}{4} M_k$ \Comment{Reduce the previous estimate of $M_{\star}$}
\Repeat
\State $M_+ \gets 2 \cdot M_+$ \Comment{Increase the estimate}
\State $\eta_+ \gets \min\bigl\{ \frac{1}{M_+}, 
\bigl[ 
\frac{p}{s2^{p - 2}\| \nabla f(x_k) \|^{p - 2}}
\bigr]^{\frac{1}{p - 1}} \bigr\}$ \Comment{Compute the step size}
\State $x_+ \gets x_k - \eta_+ \nabla f(x_k)$ \Comment{Perform the gradient step}
\Until{$f(x_k) - f(x_+) \geq \frac{\eta_{+}}{2} \|\nabla f(x_k)\|^2$}
\State $M_{k+1} \gets M_+$
\State $x_{k+1} \gets x_+$
\EndFor
\end{algorithmic}
\end{algorithm}

According to our theory, the inner loop of this algorithm is well-defined,
and the stopping condition is reached at least for $M_+ \geq M_{\star}$.
Thus, for every $k \geq 0$, bound~\eqref{BoundedMk} holds for our estimates.
Therefore, the result of Theorem~\ref{TheoremGMRate}
is immediately applicable to this procedure,
and we establish the same convergence rate as for the constant step size rule.

At the same time, denoting by $n_k$ the number of function evaluations
performed at iteration $k \geq 0$, we notice that
$M_{k + 1} = 2^{n_k - 2} M_k$.
Therefore, the total number $N_k$ of function evaluations after $k$ iterations is
$$
\ba{rcl}
N_k & = & \sum\limits_{i = 0}^{k - 1} n_i
\;\; = \;\;
\sum\limits_{i = 0}^{k - 1} \bigl[ 2 + \log_2 \frac{M_{i + 1}}{M_i} \bigr]
\;\; = \;\;
2k + \log_2 \frac{M_k}{M_0}
\;\; \overset{\eqref{BoundedMk}}{\leq} \;\;
2k + \max\{ 1 + \log_2 \frac{M_{\star}}{M_0}, \, 0 \}.
\ea
$$
Therefore, on average, our adaptive procedure performs only one extra function evaluation
per iteration, up to an additive logarithmic term.

\section{Lower Bounds for One-Step Methods}
\label{sec:one_step_lower_bounds}

In the last section, we introduced a new step size for gradient descent proving the same rate
of $\mathcal{O}(N^{-\frac{p}{p - 2}})$ for $p > 2$, where $N$ is the number of iterations, as for the composite gradient method developed in~\cite{nesterov_inexact_2022}. A natural question is whether 
this rate is the best we can achieve.

In this section, we analyze a class of \textit{one-step gradient methods},
that covers gradient descent with our step size setting and the composite gradient method~\cite{nesterov_inexact_2022}.
We establish a lower bound on the rate of convergence of these methods
that matches our upped bound, thus establishing that this rate is optimal for
one-step methods. In the next section, we will establish lower bounds
for more general classes of first-order algorithms beyond one-step methods.

Let us consider the following form of iterates,
starting with $x_0 = 0$:
\beq\label{eq:general_method}
\boxed{
\ba{rcl}
    x_{k+1} & \coloneqq & \alpha_k x_k - \beta_k \nabla q(x_k),
    \qquad k \geq 0,
\ea
}
\eeq
where $q(x) \coloneqq \frac{1}{2}x^\top A x - b^\top x$ is the quadratic part 
of our problem~\eqref{eq:prob}, and $\alpha_k, \beta_k > 0$ are some parameters.

This construction covers the following two most important cases:
\begin{enumerate}
    \item \textit{Gradient Descent}, that we studied before, fits under our framework~\eqref{eq:general_method} with the following parameterization:
\end{enumerate}
\beq \label{GradDescentParameters}
\ba{rcl}
x_{k+1} & \coloneqq & x_k - \eta_k \nabla f(x_k) \\[10pt]
& = & \underbrace{(1 - \eta_k s\|x_k\|^{p-2})}_{\alpha_k}x_k - \underbrace{\eta_k}_{\beta_k}\nabla q(x_k).
\ea
\eeq
\begin{enumerate}
    \setcounter{enumi}{1}
    \item \textit{The Composite Gradient Method}~\cite{nesterov_inexact_2022},
    which uses a linearization of
the quadratic part of~\eqref{eq:prob} but keeps the regularization term present in the model. 
For simplicity, let us fix the regularization parameter to be constant and equal to the Lipschitz constant $L$ of the gradient of $q(\cdot)$.
The composite gradient method fits our framework as follows:
\beq \label{CompositeGMParameters}
\ba{rcl}
x_{k+1} &\coloneqq& \argmin\limits_{y \in \R^d} 
\Bigl\{ 
\la \nabla q(x_k), y - x_k \ra + \frac{L}{2} \|y - x_k \|^2
+ \frac{s}{p}\|y\|^p \Bigr\} \\[15pt]
& = &
\underbrace{\textstyle \frac{L}{L + s\|x_{k+1}\|^{p-2}}}_{\alpha_k}x_k - \underbrace{\textstyle \frac{1}{L + s \|x_{k+1}\|^{p - 2}}}_{\beta_k} \nabla q(x_k).
\ea
\eeq
Note that the norm of $\|x_{k+1}\| \equiv r_{k + 1}$ can be found as the solution to the following nonlinear univariate equation:
\beq \label{CompositeNonlinearEq}
\ba{rcl}
\bigl( L + s r_{k + 1}^{p-2} \bigr) r_{k + 1} & = & \|L x_k - \nabla q(x_k)\|,
\ea
\eeq
which provides an alternative to the novel step size that we introduced in the previous section for the gradient descent.
It was established in~\cite{nesterov_inexact_2022}
that this algorithm also has a convergence rate of $\ub(k^{-p / (p - 2)})$,
similar to the one established in Theorem~\ref{TheoremGMRate}.
It has, until now, been an open question whether this rate is optimal for one-step methods. 
\end{enumerate}

Our lower bound will work for an arbitrary matrix $A \succeq 0$ with a spectrum from the given range $[\mu, L]$.
Let us denote by $v \in \R^d$ the eigenvector of unit length of matrix $A$ that corresponds to the smallest eigenvalue $\mu$. Thus, it holds
$Av = \mu v$ and $\| v \| = 1$.
We also choose an arbitrary parameter $r > 0$ and place the minimum of~\eqref{eq:prob} to be at the following point:
\beq \label{LBSol}
\ba{rcl}
\xopt & = & r v.
\ea
\eeq
To ensure~\eqref{LBSol}, it suffices to choose vector $b \in \R^d$ correspondingly as
\beq \label{BVectorChoice}
\ba{rcl}
b & := & (\mu r + sr^{p - 1}) v.
\ea
\eeq
Indeed, direct substitution of this choice into stationary condition~\eqref{StatCondition} implies that the unique solution
to our problem satisfies~\eqref{LBSol}. Note that starting our method~\eqref{eq:general_method} from $x_0 = 0$
we ensure that iterates always belong to the line spanned by $v$. Thus, we have $x_{k} \equiv r_k v$, where sequence $\{ r_k \}_{k \geq 0}$ satisfies
\beq \label{RSeq}
\ba{rcl}
r_{k + 1} & := & 
(\alpha_k - \beta_k \mu) r_k + \beta_k (\mu r + sr^{p - 1}),
\qquad k \geq 0, \qquad r_0 = 0.
\ea
\eeq
We are ready to state our main condition on the method which provides us with the lower bound on its rate of convergence. 
We denote the following modified condition number,
\beq \label{LBQDef}
\ba{rcl}
\bar{Q} & := & \frac{L + s(p - 1)r^{p - 2}}{\mu + s(p - 1) r^{p - 2}} \;\; \geq \;\; 1.
\ea
\eeq
\begin{lemma}
    Suppose we are running a general one-step method~\eqref{eq:general_method} on our construction. 
    For any $k \geq 0$, let
    \beq \label{LBCondition}
    \ba{rcl}
    r_{k} & \leq & r,
    \quad \text{and} \quad
    (r - r_{k + 1}) \;\; \geq \;\; (r - r_k)(1 - \bar{Q}^{-1}).
    \ea
    \eeq
    Then, 
    we have the following lower bound:
$$
\ba{rcl}
\|x_N - \xopt \| & \geq & r \exp\Bigl(-\frac{N}{\bar{Q} - 1}\Bigr).
\ea
$$
In particular, setting $\mu := 0$, we obtain:
\beq \label{ItersLB}
\ba{rcl}
\|x_N - \xopt \| & \geq & r \exp\Bigl(-\frac{s(p - 1)r^{p - 2} N}{L}\Bigr).
\ea
\eeq
\end{lemma}
\begin{proof}
    Indeed, due to our construction, we have
    \beq \label{eq:gm_norm_bound}
    \ba{rcl}
    \| x_{N} - x_{\star} \| & = & r - r_{N}
    \;\; \overset{\eqref{LBCondition}}{\geq} \;\;
    r (1 - \bar{Q}^{-1})^N
    \;\; \geq \;\;
    r \exp\Big(-\frac{N}{\bar{Q} - 1}\Big).
    \ea
    \eeq
\end{proof}

Let us check that both the basic gradient descent and the composite gradient method satisfy our condition.

\begin{lemma} \label{LemmaGMCondition}
    Consider iterations of gradient descent~\eqref{GradDescentParameters}
    with a step-size satisfying:
\beq \label{EtaKLBCondition}
\ba{rcl}
\eta_k & \leq & \frac{1}{L + s(p - 1) r^{p - 2}}
\;\; \equiv \;\;
\frac{1}{L + s(p - 1) \| \xopt \|^{p - 2}}.
\ea
\eeq
    Then, 
    condition~\eqref{LBCondition} holds.
\end{lemma}
\begin{proof}
Let us assume by induction that $\Delta_k := r - r_k \geq 0$.
We have, for every $k \geq 0$:
$$
\ba{rcl}
\Delta_{k + 1} & = & 
r - r_{k + 1} 
\;\; \overset{\eqref{GradDescentParameters},\eqref{RSeq}}{=} \;\;
r - r_k
- \eta_k \bigl[ \mu(r - r_k) + s r^{p - 1} - s r_k^{p - 1} \bigr] \\[10pt]
& = &
\Delta_k - \eta_k \mu \Delta_k + \eta_k s[ r_k^{p - 1} - r^{p - 1} ].
\ea
$$
Now, using convexity of the function $x \mapsto x^{p - 1}, x \geq 0$, we get
$$
\ba{rcl}
r_k^{p - 1} - r^{p - 1} & \geq & (p - 1)r^{p - 2}(r_k - r)
\;\; = \;\; -(p - 1)r^{p - 2}\Delta_k.
\ea
$$
Therefore, we obtain the following inequality,
$$
\ba{rcl}
\Delta_{k + 1} & \geq & \Delta_k \Bigl(
1 - \eta_k \bigl[ \mu + s(p - 1) r^{p - 2} \bigr] \Bigr)
\;\; \overset{\eqref{EtaKLBCondition}}{\geq} \;\;
\Delta_k \bigl( 1 - \bar{Q}^{-1} \bigr),
\ea
$$
which gives the desired bounds from~\eqref{LBCondition} for the next iteration.
\end{proof}

\begin{remark}
Note that condition~\eqref{EtaKLBCondition}
is naturally satisfied for the step-sizes~\eqref{ConstantEtaK}, \eqref{BacktrackingEtaK}
of the gradient descent
that we analyzed in the previous section.
\end{remark}

Now, we justify our condition for the composite gradient method~\eqref{CompositeGMParameters}.
For simplicity, we use a fixed Lipschitz constant $L$ as the regularization parameter.
However, the statement can also be relaxed as in the case of gradient descent.

\begin{lemma} \label{LemmaCGMCondition}
Condition~\eqref{LBCondition} holds
for iterations of the composite gradient method~\eqref{CompositeGMParameters}.
\end{lemma}
\begin{proof}
Note that each iteration of the method gives
$$
\ba{rcl}
r_{k + 1} & \overset{\eqref{CompositeGMParameters},\eqref{RSeq}}{=} &
\frac{1}{L + s r_{k + 1}^{p - 2}} \Bigl[ 
(L - \mu) r_k + \mu r + s r^{p - 1}
\Bigr],
\ea
$$
or, rearranging the terms, we get
\beq \label{CompOneStep}
\ba{rcl}
L r_{k + 1} + s r_{k + 1}^{p - 1} & = & 
(L - \mu) r_k + \mu r + s r^{p - 1}.
\ea
\eeq
As in the proof of Lemma~\ref{LemmaGMCondition}, denoting $\Delta_{k + 1} := r - r_{k + 1}$
and employing convexity of $x \mapsto x^{p - 1}, x \geq 0$, we note the following inequality,
\beq \label{ConvDelta}
\ba{rcl}
(p - 1) r^{p - 2} \Delta_{k + 1} & \geq & r^{p - 1} - r_{k + 1}^{p - 1}.
\ea
\eeq
Therefore, we have that
$$
\ba{rcl}
L \Delta_{k + 1} + s(p - 1) r^{p - 2} \Delta_{k + 1}
& \overset{\eqref{ConvDelta}}{\geq} &
L \Delta_{k + 1} + s r^{p - 1} - s r_{k + 1}^{p - 1} 
\;\; \overset{\eqref{CompOneStep}}{=} \;\;
(L - \mu) \Delta_k.
\ea
$$
Hence, we obtain
$$
\ba{rcl}
\Delta_{k + 1} & \geq & \Delta_k \frac{L - \mu}{L + s(p - 1) r^{p - 2}} 
\;\; = \;\;
\Delta_k \bigl(1 - \bar{Q}^{-1} \bigr),
\ea
$$
which completes the proof.
\end{proof}

Note that in our construction we do not use the specific structure of other eigenvalues of $A$ except $v$, which corresponds to the smallest eigenvalue: their choice can be arbitrary. To be specific, we can fix an arbitrary $v \in \R^d$ s.t. $\|v\| = 1$ and set
$$
\ba{rcl}
A & := & (I - v v^{\top}) L,
\ea
$$
where $L > 0$ is given. The corresponding vector $b$ is chosen as before by~\eqref{BVectorChoice},
where $r > 0$ is fixed. It is easy to verify that this instance of our problem satisfy all the desired properties discussed above, in particular, we have $\| \xopt \| = r$.

We are ready to establish the lower bound
for one-step gradient methods.

\begin{theorem}\label{thm:gm_final_lb}
Let $s > 0$ and $L > 0$. For any one-step method such that~\eqref{LBCondition} holds, running on our
instance of the regularized quadratic problem from $x_0 = 0$, the functional residual of the last iteration is lower bounded by
\beq
\ba{rcl}
        f(x_N) - f_\star 
        &=& \lb \left( \Big[\frac{L}{s^{2/p} N}\Big]^{\frac{p}{p-2}} \right),
\ea
\eeq
and the gradient norm is bounded as
\beq
    \ba{rcl}
    \| \nabla f(x_N) \| 
    &=&
    \lb \left( \Big[\frac{L}{s^{1/(p - 1)} N}\Big]^{\frac{p-1}{p-2}} \right).
    \ea
\eeq
\end{theorem}
\begin{proof}
Note that $r > 0$ is a free parameter of our construction, 
which means that we can maximize the lower bounds with respect to 
it. 

Let us look at the simplified version of our lower bound on the norm distance for $\mu := 0$, provided by~\eqref{ItersLB}.
We denote
$$
\ba{rcl}
\mathcal{L}(r) 
    &\coloneq& 
    r\exp\left(-\frac{s (p - 1) r^{p-2} N}{L}\right),
\ea
$$
and consider its extremum values.
The first-order optimality condition gives
$$
\ba{rcl}
\mathcal{L}'(r) & = &
\exp \left(-\frac{s(p - 1)r^{p-2}N}{L}\right)
\left[1 - s(p - 1)(p - 2) \frac{sN}{L} \, r^{p-2}\right] \;\; = \;\; 0,
\ea
$$
which, taking into account $r > 0$, implies that
the optimum value is
$$
\ba{rcl}
r^{\star} & = & \left[ \, \frac{1}{(p - 1)(p - 2)} \cdot \frac{L}{sN} \, \right]^{\frac{1}{p - 2}}
\;\; \propto \;\; \Bigl( \frac{L}{sN}\Bigr)^{\frac{1}{p-2}}.
\ea
$$
This expression suggests the right choice for the value of $r$.
Now, let us simply set
\beq \label{RChoice1}
\boxed{
\ba{rcl}
r & := & \Bigl( \frac{L}{s(p - 1)N}  \Bigr)^{\frac{1}{p - 2}}.
\ea
}
\eeq
Then, 
\beq \label{OneStepDistanceLB}
\ba{rcl}
\| x_N - \xopt \| & \overset{\eqref{ItersLB}}{\geq} &
\mathcal{L}(r)
\;\; \overset{\eqref{RChoice1}}{=} \;\;
\frac{1}{e} \Bigl( \frac{L}{s(p - 1)N}  \Bigr)^{\frac{1}{p - 2}}.
\ea
\eeq
Using uniform convexity of our objective, we obtain
\beq \label{FuncResBelow}
\ba{rcl}
f(x_N) - f_{\star}
& \overset{\eqref{eq:func_residual}, \eqref{eq:sigma_value}}{\geq} &
\frac{s}{p 2^{p - 2}} \|x_N - \xopt\|^p
\;\; \overset{\eqref{OneStepDistanceLB}}{\geq} \;\;
\frac{s}{p 2^{p - 2} e^p}
\Bigl( \frac{L}{s(p - 1)N}  \Bigr)^{\frac{p}{p - 2}}
\ea
\eeq
and
$$
\ba{rcl}
\| \nabla f(x_N) \|^{\frac{p}{p - 1}}
& \overset{\eqref{eq:grad_ineq}, \eqref{eq:sigma_value}}{\geq} &
\frac{p}{p - 1}
(s 2^{2 - p})^{\frac{1}{p - 1}}
( f(x_N) - f_{\star} )
\;\; \overset{\eqref{FuncResBelow}}{\geq} \;\;
\frac{(s 2^{2 - p})^{\frac{p}{p - 1}}}{(p - 1) e^p} 
\Bigl( \frac{L}{s(p - 1)N}  \Bigr)^{\frac{p}{p - 2}},
\ea
$$
which completes the proof.
\end{proof}

\section{Lower Bounds for General First-Order Methods} \label{sec:lower_bounds}

\subsection{Trajectories of First-order Methods}

In this section, let us study the trajectory
of any first-order method
when applied to our regularized quadratic problem~\eqref{eq:prob}.
We start with a formal notion of the \textit{first-order black-box oracle} $\mathcal{I}_f$, 
originating from \cite{nemirovski_problem_1983}.
For any given point, the oracle returns local information
about the objective function of the first order, which is 
the function value and the gradient vector:
$$
\ba{rcl}
    \mathcal{I}_f(x) & \coloneqq & 
    \big\{f(x), \nabla f(x)\big\},
    \qquad x \in \R^d.
\ea
$$
Then, any
first-order method $\mathcal{M}$
can be uniquely characterized by
a sequence of mappings:
$
\mathcal{M} =
(\mathcal{M}_0, \mathcal{M}_1, \mathcal{M}_2, \ldots).
$
Each mapping $\mathcal{M}_k$ defines the next point $x_{k + 1}$
of the algorithm, 
given all information known prior to the current iteration:
$$
\ba{rcl}
x_{k+1} & \coloneq & 
\mathcal{M}_k\big( \mathcal{I}_f(x_0), \, \mathcal{I}_f(x_1), 
\, \ldots, \, \mathcal{I}_f(x_k)\big),
\qquad k \geq 0.
\ea
$$
We assume that $x_0 = 0$, and that the number of iterations $N \geq 1$
is fixed. Then, without loss of generality, we always assume that the last iterate $x_N$ is 
\textit{the result of the method} when applied to a particular objective $f$.
For a fixed method $\mathcal{M}$,
we will design a resisting objective function from our class~\eqref{eq:prob},
which will provide us with the lower bound for the performance of any
first-order method.

Let us introduce the following matrix:
\begin{equation} \label{AStarDef}
\ba{rcl}
    \Ar & \coloneq & A + s\|\xopt\|^{p-2}I,
\ea
\end{equation}
where $\xopt$ is the unique solution to our problem~\eqref{eq:prob}.
Note that the first-order stationarity condition provides us 
with the following expression for the solution:
\beq \label{XStarExpr}
\ba{rcl}
\xopt & = & \Ar^{-1} b.
\ea
\eeq

An object that will be central to our analysis is the notion of the \textit{$i$-th order Krylov subspace} $\kryl_i(f)$, where $i \geq 0$, which is associated with regularized quadratic functions of the form~\eqref{eq:prob}. We define:
\begin{equation} \label{KrylovDef}
\ba{rcl}
    \kryl_i(f) & \coloneq & \spn\{b, 
    \, \Ar b, \, \Ar^2 b, \, \ldots, \, \Ar^{i-1} b\}.
\ea
\end{equation}
Due to the structure of the matrix $\Ar$, it is easy to see that $\kryl_i(f)$
are equal to the standard Krylov subspaces associated with
matrix $A$ and vector $b$, \emph{i.e.}, we have
$\kryl_i(f)  =  \spn\{b, A b, A^2 b, \ldots, A^{i-1} b\}$.

Note that the following inclusion clearly holds: 
$$
\ba{rcl}
\kryl_i(f) & \subseteq & \kryl_{i + 1}(f), \qquad i \geq 0.
\ea
$$
Moreover, if one of these
inclusions is an equality (i.e., $\kryl_i(f) = \kryl_{i + 1}(f)$), 
then so are the following.
This implies that the preceding inclusions are strict until
we arrive at a subspace invariant to further applications of $\Ar$.
Hence, we obtain the following chain:
$$
\ba{rcl}
\kryl_0(f) & \subset &
\kryl_1(f) \;\; \subset \;\;
\;\; \ldots \;\; \subset \;\;
\kryl_{\ell}(f) \;\; = \;\; \kryl_{\ell + 1}(f) 
\;\; = \;\;
\ldots
\ea
$$
The dimension of the subspace at index $\ell$ where the following Krylov subspaces are equal is commonly referred to as the \textit{maximal dimension} of the sequence of Krylov subspaces. 

The classic Cayley-Hamilton theorem \cite{hoffman_linear_1971}, implies that $\xopt = \Ar^{-1} b$ can be represented in a polynomial form: $\xopt = c_0b + c_1\Ar b + \ldots + c_{d - 1} \Ar^{d-1}b$, for some sequence of real coefficients $c_0, \ldots, c_{d-1} \in \R$.
Therefore, 
\beq \label{XoptKrylov}
\ba{rcl}
\xopt & = & \Ar^{-1} b \in \kryl_{\ell}(f).
\ea
\eeq

We will construct an instance of problem~\eqref{eq:prob},
using the spectral representation 
\beq \label{SpecDec}
\ba{rcl}
A & = & U \Lambda U^{\top},
\ea
\eeq
where $\Lambda = \Diag(\lambda_1, \ldots, \lambda_d)$, and $U$ is an orthogonal matrix:
$UU^{\top} = I$.
The matrix $\Lambda$
is our \textit{fixed parameter}, while the matrix $U$ will be constructed adversarially to ensure that the trajectory of the method lies in the Krylov subspaces.
This construction is classical, and it was developed initially for quadratic problems in the works of
Nemirovski~\cite{arkadi_nemirovski_information-based_1994}, and adapted for analysis on functions of form~\eqref{eq:prob} in an unpublished master's thesis~\cite{master_thesis}.
For the sake of completeness, we provide proofs of our statements in what follows. 

For a fixed arbitrary $r > 0$ and a vector $\pi \in \Delta_d$ 
from the standard simplex,
we set
\beq \label{BChoice}
\ba{rcl}
b & := & r( \Lambda + s r^{p - 2} I ) \sqrt{\pi} \;\; \in \;\; \R^d,
\ea
\eeq
where $\sqrt{\pi} \in \R^d$ is the vector whose
entries are square roots of the corresponding entries of $\pi$: $[ \sqrt{\pi} \, ]_i \equiv \sqrt{\pi_i}$.
This choice of $b$ is justified by the following simple lemma.

\begin{lemma} \label{LemmaXOpt}
For any $\lambda_1, \ldots, \lambda_d \geq 0$, $r > 0$, and $\pi \in \Delta_d$,
let $\Lambda := \Diag(\lambda_1, \ldots, \lambda_d)$, and
$b \in \R^d$ be chosen according to \eqref{BChoice}.
Then, for an arbitrary orthogonal $b$-invariant matrix $U$,
we have that the solution $\xopt$ to the corresponding
regularized quadratic problem is
\beq \label{Xopt}
\ba{rcl}
\xopt & = & rU \sqrt{\pi},
\ea
\eeq
and, consequently,
\beq \label{XoptR}
\ba{rcl}
\| \xopt \| & = & r.
\ea
\eeq
\end{lemma}
\begin{proof}
Using the first-order stationarity condition
and uniform convexity, we know that the solution $\xopt$
to~\eqref{eq:prob}
is the unique solution of the following nonlinear equation,
\beq \label{StatCond}
\ba{rcl}
A \xopt + s\|\xopt\|^{p - 2} \xopt
& = & b.
\ea
\eeq
Now, let us consider the vector 
$$
\ba{rcl}
\bar{x} & := & (A + sr^{p - 2} I)^{-1} b
\;\; = \;\; U( \Lambda + sr^{p - 2} I  )^{-1}U^{\top}b \\
\\
& \overset{(*)}{=} & U( \Lambda + sr^{p - 2} I  )^{-1}b 
\;\; \overset{\eqref{BChoice}}{=} \;\; 
r U \sqrt{\pi},
\ea
$$
where we used in $(*)$ that $U$ is $b$-invariant, thus
$U^{\top}b = b$.
Hence, $\| \bar{x} \|^2 = r^2 \| \sqrt{\pi} \|^2 = r^2$,
and we conclude that $\xopt = \bar{x}$,
which completes the proof.
\end{proof}

\begin{corollary}
    For any choice of the parameter $r > 0$,
    there exists a regularized quadratic
    function from our family, such that
    $\| \xopt \| = r$.
\end{corollary}

Under suitable conditions, our choice of $b$ implies that the associated Krylov subspaces are strictly nested in one another until the Krylov subspace is of the $d$-th order
(thus covering the whole space). We formalize this notion in the following lemma:
\begin{lemma}\label{LemmaKrylovStrict}
    Let $\lambda_d > \ldots > \lambda_2 > \lambda_1 \geq 0$. Then, for $A = \Diag(\lambda_1, \ldots, \lambda_d)$ and $b$ given by~\eqref{BChoice}, we have
\beq \label{KrylInclusion}
\ba{rcl}
\kryl_0(f) & \subset &
\kryl_1(f) \;\; \subset \;\;
\;\; \ldots \;\; \subset \;\;
\kryl_{\ell}(f) \;\; = \;\; \kryl_{\ell + 1}(f) 
\;\; = \;\;
\ldots
\ea
\eeq
and
\beq \label{XoptInclusion}
\ba{rcl}
\xopt & \in & \kryl_{\ell}(f),
\ea
\eeq
where $\ell$ is the number of non-zero components in the vector $b$.
\end{lemma}
\begin{proof}
    Consider the following \textit{Krylov matrix} associated with a $d$-th order Krylov subspace:
    $$
    \ba{rcl}
    K_{d}(f) & = & \Diag(b_1, \ldots, b_d) V \;\; \in \;\; \R^{d \times d},
    \ea
    $$
    where $b_i = r \sqrt{\pi_i} (\lambda_i + sr^{p-2})$ are the entries of the vector $b$, and $V$ is the Vandermonde matrix:
    $$
    \ba{rcl}
        V & = & \begin{pmatrix}
            1 & \lambda_1 & \cdots & \lambda_1^{d-1} \\
            1 & \lambda_2 & \cdots & \lambda_2^{d-1} \\
            \vdots & \vdots & \ddots & \vdots \\
            1 & \lambda_d & \cdots & \lambda_d^{d-1}
        \end{pmatrix}.
    \ea
    $$
    It is known that the Vandermonde matrix has full rank if $\lambda_i \not= \lambda_j, \forall i, j$. 
    Clearly, $\Diag(b_1, \ldots, b_d)$ has rank equal to the number of nonzero components of $b$. Consequently, the rank of $K_d(f)$ is equal to the rank of $\Diag(b_1, \ldots, b_d)$, which we denote by $\ell$. 
    The $d$-th order Krylov subspace is the column space of $K_d(f)$, and as a result, we have $\dim \kryl_d(f) = \ell$. The whole chain in~\eqref{KrylInclusion} now follows from the fact that
    the dimension of every next Krylov subspace can be increased no more than by one, and $\dim \kryl_0(f) = 0$.
    Therefore, all Krylov subspaces of lower orders are strictly nested within $\kryl_{\ell}(f)$, and the higher order subspaces are all equal.
\end{proof}

\begin{remark}
Note that Lemma~\ref{LemmaKrylovStrict} extends to any matrix $A = 
U \Lambda U^{\top}$ with
orthogonal $b$-invariant matrix $U$,
due to the simple observation:
$$
\ba{rcl}
\spn\{b, \, U\Lambda U^\top b, \, \ldots, \, (U\Lambda U^\top)^k b\} 
& = & U \spn\{b, \, \Lambda b, \, \ldots, \, \Lambda^k b\}.
\ea
$$
Therefore, under the same conditions, 
we have the strict inclusion of the Krylov subspaces
for any regularized quadratic, whose matrix $A$ is of the form in~\eqref{SpecDec}
with $b$-invariant $U$.
\end{remark}

The following lemma is the main tool of analyzing
the trajectory of any first-order method
when applied to our regularized quadratic function.

\begin{lemma}\label{lem:kryltraj}
Let parameters $p \geq 2$, $s \geq 0$,
$\lambda_d > \ldots > \lambda_1 \geq 0$, $r > 0$, and $\pi \in \Delta_d$  
be fixed.
Let $\ell$ be the maximal dimension of the
sequence of Krylov subspaces.
For any method $\mathcal{M}$ that
performs $N \leq \frac{\ell}{2}$ iterations,
there exists a regularized quadratic function
from our family,
such that the full trajectory
$\{x_k\}_{k = 0}^N$ of the method when applied 
to $f$ belongs to the Krylov subspaces:
\beq \label{KrylIterates}
\ba{rcl}
x_k & \in & \kryl_{2k}(f), \qquad 0 \leq k \leq N. 
\ea
\eeq
\end{lemma}
\begin{proof}
We construct a sequence of functions $f_0, \ldots, f_N$,
each of the form
\beq \label{FkDef}
\ba{rcl}
f_k(x) & := & \frac{1}{2} x^{\top} A_k x - b^{\top} x + \frac{s}{p}\|x\|^p,
\ea
\eeq
where $A_k = U_k \Lambda U_k^{\top}$,
and $\{ U_k \}_{k \geq 0}$ is a sequence of orthogonal $b$-invariant matrices,
generated in an adversarial way in response to the iterations of the method.
We start with $U_0 := I$.
Let us denote by $\xopt^{(k)}$ the solution to \eqref{FkDef},
which according to Lemma~\ref{LemmaXOpt} has the following explicit representation:
$$
\ba{rcl}
\xopt^{(k)} & = & r U_k \sqrt{\pi}.
\ea
$$
Thus, all $\xopt^{(k)}$ have the same length: $\| \xopt^{(k)} \| = r$.

We are going to ensure the following key property: any pair of
functions $f_k$ and $f_{\ell}$
for $\ell > k$ are
\textit{informationally indistinguishable}
by the first-order oracle
during the first $k$ iterations of the method, \emph{i.e.,}
\beq \label{OracleEq}
\ba{rcl}
f_k(x_i) & = & f_{\ell}(x_i) 
\quad \text{and} \quad 
\nabla f_k(x_i) \;\; = \;\; \nabla f_{\ell}(x_i),
\qquad 0 \leq i \leq k,
\ea
\eeq
where $\{ x_i \}_{i = 0}^k$ are iterations of $\mathcal{M}$
when applied to the function $f_k$.
In other words, the method cannot recognize whether it runs on $f_k$
or on $f_{\ell}$, $\ell > k$, when performing the first $k$ steps.

Let us verify by induction the following statement:
for every $0 \leq k \leq N$, we have
\beq \label{KrylovProperty}
\ba{rcl}
x_0, \ldots, x_k & \in & \kryl_{2k}(f_k).
\ea
\eeq
Then, we will choose $f \equiv f_N$, and \eqref{KrylIterates}
will follow immediately.

Inclusion~\eqref{KrylovProperty} trivially holds for $k = 0$,
since we always choose $x_0 := 0$, and, by definition, $\kryl_0 \equiv \{ 0 \}$.
Now, assume that~\eqref{KrylovProperty} holds for some $k \geq 0$,
and consider one step of our resisting strategy.

Due to our assumption that $k + 1 \leq N \leq \ell / 2$, by Lemma~\ref{LemmaKrylovStrict} we have a strictly increasing sub-chain
of Krylov subspaces:
$$
\ba{rcl}
\kryl_{2k}(f_k) 
& \subset & 
\kryl_{2k + 1}(f_k)
\;\; \subset \;\;
\kryl_{2k + 2}(f_k).
\ea
$$
Let us build an orthogonal matrix $H$ such that 
it is identical on $\kryl_{2k + 1}(f_k)$:
$H|_{\kryl_{2k + 1}(f_k)} \equiv I$,
and $x_{k + 1} \in H \kryl_{2k + 2}(f_k)$.
We can construct such a matrix explicitly, as follows.
First, we choose an orthonormal basis
$v_1, \ldots, v_{2k + 1} \in \R^d$ in $\kryl_{2k + 1}(f_k)$
and complement it by a vector $v_{2k + 2} \in \R^d$
to have an orthonormal basis in $\kryl_{2k + 2}(f_k)$.
Consider the corresponding projector onto $\kryl_{2k + 1}(f_k)$:
$$
\ba{rcl}
\Pi & := & \sum\limits_{i = 1}^{2k + 1} v_i v_i^{\top}.
\ea
$$
Then, the vector $x_{k + 1}$ can be decomposed into a direct sum
of two vectors:
$
x_{k + 1} = \Pi x_{k + 1} + (I - \Pi) x_{k + 1}.
$
Let $H \in \R^{d \times d}$ be a Householder reflection
that transforms the vector $v_{2k + 2}$ 
into the vector $y := \frac{(I - \Pi) x_{k + 1}}{\| (I - \Pi) x_{k + 1}\|}$,
defined by the following formula:
$$
\ba{rcl}
H & := & I - 2 w w^{\top}, \qquad \text{where} \qquad
w \;\; := \;\; \frac{v_{2k + 2} - y}{\|v_{2k + 2} - y \|}.
\ea
$$
It is immediate to check that such $H$ satisfies the required properties,
and we set $U_{k + 1} := H U_k$.

Let us verify that $f_{k}$ and $f_{k + 1}$
remain informationally indistinguishable
for the first $k$ iterates of the method.
Indeed, we have $A_{k + 1} = H A_k H^\top$,
and due to our inductive assumption~\eqref{KrylovProperty},
$$
\ba{rcl}
A_{k + 1} x_i & \equiv & A_k x_i, \qquad 0 \leq i \leq k.
\ea
$$
Therefore, \eqref{OracleEq} is satisfied.
It remains to check~\eqref{KrylovProperty},
which follows from the following representation:
$$
\ba{rcl}
\kryl_{2(k + 1)}(f_{k + 1}) & = &
H \kryl_{2(k + 1)}(f_{k}).
\ea
$$
This completes the proof.
\end{proof}

As a consequence,
we can represent the result of the method,
when applied to our constructed function~$f$ in the following 
form:
\begin{equation}\label{eq:poly_rep}
    x_N = q(\Ar) b,
\end{equation}
where $q \in \mathcal{P}_{2N}$ is a $2N$-degree matrix polynomial.
This reduction allows us to state the following proposition,
which is the main result used in the next section.

\begin{proposition}\label{prop:poly}
    Let $\mathcal{M}$ be a first-order method that performs $N$ iterations.
    Let $\Lambda = \Diag(\lambda_1, \ldots, \lambda_d)$, 
    where $\lambda_d > \ldots > \lambda_2 > \lambda_1 \geq 0$, 
    and $b$ be given by~\eqref{BChoice}, such that $\pi \in \Delta_d$. 
    Then, for any $p \geq 2$ and $s \geq 0$, there exists a regularized convex 
    quadratic function $f$ of the form~\eqref{eq:prob} with the given parameters, 
    such that the distance between the solution $\xopt$ to $f$ and the result 
    of the method $x_N$, when applied to $f$, is bounded as:
    \begin{equation} \label{SolutionGap}
    \ba{rcl}
        \|x_N - \xopt\| & \geq &
        \min\limits_{q \in \mathcal{P}_{2N}} \|q(\Ar)b - \xopt\|.
    \ea
    \end{equation}
\end{proposition}
\begin{proof}
    We divide the proof of this proposition into two cases. 
    Let us first consider the case when $2N \leq \ell$, where $\ell$ is the the maximal dimension 
    of the sequence of Krylov subspaces. Then we use Lemma~\ref{lem:kryltraj} to create an instance of a regularized quadratic function $f$ such that $x_N \in \kryl_{2N} (f)$. This allows us to write $x_N = q(\Ar)b$, for some $q \in \mathcal{P}_{2N}$. For this specific polynomial, we have
    \begin{equation*}
    \ba{rcl}
        \|x_N - \xopt\| & = & \|q(\Ar)b - \xopt\|,
    \ea
    \end{equation*}
    and the statement follows from minimizing the right-hand side over the space of $2N$-degree polynomials.

    Let us now consider the case $2N > \ell$. We already know that $\xopt \in \kryl_{\ell}(f)$ by~\eqref{XoptKrylov}. We can therefore select a $2N$-degree polynomial $q$ such that $\xopt = q(\Ar) b$, and~\eqref{SolutionGap} becomes the trivial lower bound:
    $$
    \ba{rcl}
        \|x_N - \xopt \| \;\; \geq \;\; \| q(\Ar) b - \xopt \| & \equiv & 0,
    \ea
    $$
    which completes the proof.
\end{proof}

\subsection{Complexity Lower Bound}
In this subsection, we use the fact that Lemma~\ref{LemmaXOpt} leaves us a degree of freedom to choose the norm of the solution $\| \xopt \|$ for the construction. We assume that $\| \xopt \| > 0$, since we always start at $x_0 = 0$, and therefore if $\xopt = 0$
the solution would be found in $0$ iterations.

In order to prove our lower bound,
we will also need the following technical result, which
characterizes the solution to the problem of the best uniform polynomial approximation.
We restate it here without proof.
\begin{lemma}\label{lem:min_max_sol}
        Let $n \geq 1$, $0 < a \leq b$ and $c = b / a$. Then
        \beq \label{eq:approx_solution}
        \ba{rcl}
            \min\limits_{q \in \mathcal{P}_n} 
            \max\limits_{a \leq t \leq b} |1 - tq(t)| 
        &=& \Theta(c)
        \;\; \coloneq \;\;
        2 \left[\Big(\frac{\sqrt{c} + 1}{\sqrt{c} - 1}\Big)^{n} + \Big(\frac{\sqrt{c} - 1}{\sqrt{c} + 1}\Big)^{n}\right]^{-1},
        \ea
        \eeq 
        where,
        \beq \label{eq:min_max_sol_lb}
        \ba{rcl}
        \Theta(c) & \geq & \Big[\exp\Big(\frac{2n}{\sqrt{c} - 1}\Big) + 1\Big]^{-1}.
        \ea
        \eeq
        Moreover, there exist a finite grid $a \leq t_1 < \ldots < t_{n+1} \leq b$, 
        and a set of coefficients $\pi_1, \ldots, \pi_{n+1}$ from the standard simplex $\Delta_{n+1}$ such that
        \begin{equation} \label{eq:discretization}
        \ba{rcl}
            \min\limits_{q \in \mathcal{P}_n} 
            \sum\limits_{i=1}^{n+1} \pi_i[1 - t_i q(t_i)]^2 
            & = & \Theta(c)^2.
        \ea
        \end{equation}
\end{lemma}
Note that~\eqref{eq:approx_solution} is a classical result from the field of approximation theory \cite{mason_chebyshev_2002, cheney_introduction_1998}. The problem on the left-hand-side is equivalent to the normalized uniform approximation problem
\begin{equation*}
    \min_{\substack{u \in \mathcal{P}_{n+1} \\ \text{s.t.}\, u(0) = 1}} 
    \max_{a \leq t \leq b}|u(t)|,
\end{equation*}
and it is known that the solution to this problem is the polynomial
$u^{\star}_n(t) = \Theta(c) \cdot T_n\left(\frac{c + 1 - \frac{2}{a}t}{c - 1}\right)$,
where $T_n$ is the Chebyshev polynomial of the first kind \cite{mason_chebyshev_2002, arkadi_nemirovski_information-based_1994, carmon_analysis_2018}. 
Equation~\eqref{eq:discretization} similarly follows by the fact that the solution to \eqref{eq:approx_solution} is given by Chebyshev polynomials. These have $n+1$ extrema of the same magnitude in the given range, allowing us to represent the problem using a convex combination of these $n+1$ points~\cite{mason_chebyshev_2002, carmon_analysis_2018}. See Lemma~4 of \cite{carmon_analysis_2018} for a complete proof of these results. 

We are now in a position where we can prove a general complexity lower bound using our construction on problems of the form~\eqref{eq:prob}:
\begin{theorem}\label{thm:main_thm}
Let $p \geq 2$, $s > 0$, $0 \leq \mu \leq L$ and fix $r > 0$. For any first-order method running for $N \leq \frac{d - 1}{2}$ iterations, there is a regularized quadratic function $f$ such that $\|\xopt\| = r$,
$\lambda_{\min}(A) = \mu$, $\lambda_{\max}(A) = L$, 
and
\beq \label{thm:main_dist_eq}
\ba{rcl}
    \|x_N - \xopt\| & \geq & 
    r \exp \left(-\frac{8N}{\sqrt{Q_\star} - 1}\right),
\ea
\eeq
where
$$
\ba{rcl}
    Q_\star & \coloneq & \frac{L + s r^{p-2}}{\mu + sr^{p-2}}
    \;\; \geq \;\; 1,
\ea
$$
    is a modified condition number of the associated problem.
Furthermore, the functional residual of the last iteration satisfies the lower bound:
\beq \label{eq:main_thm}
\ba{rcl}
    f(x_N) - f_\star & \geq & \frac{s}{p2^{p-2}}r^p \exp\Big(-\frac{8pN}{\sqrt{\Qmod} - 1}\Big),
\ea
\eeq
and the gradient norm is bounded as:
\beq \label{eq:grad_norm_lb}
\ba{rcl}
    \|\nabla f(x_N)\| 
    &\geq& 
    \frac{2^{2-p}s}{(p - 1)^{(p - 1) / p}} r^{p - 1} 
         \exp\Big(- 
 \frac{8(p - 1)N}{\sqrt{\Qmod} - 1}\Big).
\ea
\eeq
\end{theorem}
\begin{proof}
    Let $\mathcal{M}$ be a fixed first-order method running for $N$ iterations, 
    where $2N + 1 \leq d$. Furthermore, let $A$ be of the form in~\eqref{SpecDec}, where $\mu = \lambda_1 < \lambda_2 < \ldots < \lambda_d = L$, and we construct $U$ in accordance with Lemma~\ref{lem:kryltraj}. This choice of $U$ lets us represent the last iteration using a polynomial in $\Ar = A + s\|\xopt\|^{p-2}I = A + sr^{p-2}I$, with eigenvalues $\lambda^{\star}_i = \lambda_i + s r^{p-2}$. We continue defining $b$ as in~\eqref{BChoice}, leaving us the choice of the coefficients $\pi_i$ and the eigenspectrum of $A$ before we have a fully-defined construction of a regularized quadratic function. Using Proposition~\ref{prop:poly} and Lemma~\ref{LemmaXOpt} we get
    \begin{equation}\label{eq:construction_invoked}
    \ba{rcl}
        \|x_N - \xopt\|^2 &\geq& \min\limits_{q \in \mathcal{P}_{2N}}\|q(\Ar)b - \xopt\|^2 
        \;\; = \;\; \min\limits_{q \in \mathcal{P}_{2N}}\|[q(\Ar) \Ar - I] \xopt\|^2 \\
        \\
        &\overset{(\ref{Xopt})}{=}& \min\limits_{q \in \mathcal{P}_{2N}}\|[q(\Ar) \Ar - I] U r \sqrt{\pi}\|^2 
        \;\; = \;\; r^2 \min\limits_{q \in \mathcal{P}_{2N}} \sum\limits_{i=1}^d \pi_i [1 - \lambda_i^{\star} q(\lambda_i^{\star})]^2,
    \ea
    \end{equation}
    which holds for any $\pi \in \Delta_d$ on the standard simplex. 
    Note that the choice of $\pi$ is arbitrary, and that for any $\pi' \in \Delta_k$, where $k\leq d$, there is some $\pi$ such that the weighted sum over $\pi$ in equation~\eqref{eq:construction_invoked} can be lower-bounded by a weighted sum over $\pi'$\footnote{For example, setting $\pi_{i} = \pi'_{i}$ for all $i \leq k$, and $\pi_i = 0$ for $i>k$.}:
    \begin{equation*}
    \ba{rcl}
        \|x_N - \xopt\|^2 
        & \geq &
        r^2
        \min\limits_{q \in \mathcal{P}_{2N}} \sum\limits_{i=1}^d \pi_i [1 - \lambda_i^{\star} q(\lambda_i^{\star})]^2 
        \;\; \geq \;\;
        r^2
        \min\limits_{q \in \mathcal{P}_{2N}} 
        \sum\limits_{i=1}^{2N + 1} \pi'_i [1 - \lambda_i^{\star} q(\lambda_i^{\star})]^2.
    \ea
    \end{equation*}
    Now, we can simply use the $\pi'$ from Lemma~\ref{lem:min_max_sol} with the notation $\mu_\star \coloneq \mu + sr^{p-2}$, $L_\star \coloneq L + sr^{p-2}$, and $\Qmod \coloneq L_\star / \mu_\star$. Taking the square root, we obtain the following bound:
    \beq \label{XNBound}
    \ba{rcl}
        \|x_N - \xopt\| 
        & \geq & 
        r \min\limits_{q \in \mathcal{P}_{2N}} 
        \max\limits_{\mu_\star \leq t \leq L_\star} |1 - t q(t)| 
        \;\; \stackrel{\eqref{eq:approx_solution}}{=} \;\; 2r \left[\Big(\frac{\sqrt{Q_\star} + 1}{\sqrt{Q_\star} - 1}\Big)^{2N} + \Big(\frac{\sqrt{Q_\star} - 1}{\sqrt{Q_\star} + 1}\Big)^{2N}\right]^{-1} \\
        \\
        & \stackrel{\eqref{eq:min_max_sol_lb}}{\geq} & 2r \Big[\exp\Big(\frac{4N}{\sqrt{\Qmod} - 1}\Big) + 1\Big]^{-1}
        \;\; \geq \;\; r \exp\Big(-\frac{8N}{\sqrt{\Qmod} - 1}\Big),
    \ea
    \eeq
    completing the first part of the proof. 
    The bounds for the functional residual and the gradient norm follows directly from uniform convexity
    (see the end of Theorem~\ref{TheoremGMRate}).
\end{proof}

\begin{remark} \label{RemarkSimplified2}
    Note that setting $\mu = 0$ and 
    ensuring that $\Qmod \geq 2$,
    we obtain the following simpler expressions for the lower bounds:
    \beq \label{RemNormDist2}
    \ba{rcl}
    \|x_N - \xopt\| & \geq & 
    r \exp\left( -16N \sqrt{\frac{sr^{p-2}}{L}} \right)
    \ea
    \eeq
    
    \beq \label{RemFuncRes2}
    \ba{rcl}
    f(x_N) - f_{\star} & \geq & 
    \frac{s}{p2^{p-2}} r^p \exp\left(-16pN \sqrt{\frac{sr^{p-2}}{L}} \right),
    \ea
    \eeq
    and
    \beq \label{RemGradRes2}
    \ba{rcl}
    \| \nabla f(x_N) \| & \geq & 
    s 2^{2-p}\Big(\frac{1}{p-1}\Big)^{\frac{p-1}{p}} r^{p - 1} 
         \exp\left(-16(p - 1)N \sqrt{\frac{sr^{p - 2}}{L}} \,\right).
    \ea
    \eeq
\end{remark}

\subsection{Implications for Specific Classes of Functions}\label{sec:implications}

\paragraph{Uniformly Convex Functions} %
Throughout this section, we assume that $p > 2$ and $s > 0$.
We can set the smallest eigenvalue 
$\mu = \lambda_{\min}(A)$ to be $0$. In this scenario, our function is uniformly convex, but not strongly convex. 

We now have the tools to state and prove our main result on the lower bound of minimizing regularized convex quadratic functions using first-order methods:
\begin{theorem}\label{thm:final}
    Let $p > 2, s > 0$ and $L > 0$. 
    For any first-order method
    that performs $1 \leq N \leq \frac{d - 1}{2}$ iterations, 
    there exists a regularized quadratic function $f$ 
    with $\lambda_{\max}(A) = L$,
    such
    that the functional residual is bounded as:
    \beq \label{FuncResFinal}
    \ba{rcl}
        f(x_N) - f_\star 
        &=& \lb \left( \Big[\frac{L}{s^{2/p} N^2}\Big]^{\frac{p}{p-2}} \right).
    \ea
    \eeq
    And for minimizing the gradient norm, we have the following lower bound:
    \beq \label{GradNormFinal}
    \ba{rcl}
    \| \nabla f(x_N) \| 
    &=&
    \lb \left( \Big[\frac{L}{s^{1/(p - 1)} N^2}\Big]^{\frac{p-1}{p-2}} \right).
    \ea
    \eeq
\end{theorem}
\begin{proof}
    After using Theorem~\ref{thm:main_thm}, the proof is similar to that of Theorem~\ref{thm:gm_final_lb}, maximizing the right hand side expression with respect to $r$. This leads us to setting
    \beq
    \ba{rcl}
    r & = & \Bigl(\frac{L}{3sN^2}\Bigr)^{\frac{1}{p-2}},
    \ea
    \eeq
    which also ensures that $Q_{\star} =  3N^2 + 1 \geq 4$.
    Inserting this expression for $r$ into lower bounds~\eqref{RemFuncRes2} and~\eqref{RemGradRes2} completes the proof.
\end{proof}

Note that lower bound~\eqref{FuncResFinal}
for the functional residual is
the same as it was established in \cite{doikov_lower_2022} 
for a general smooth uniformly convex functions, with a non-quadratic worst-case
instance of the resisting function.
Therefore, our result confirms that \textit{regularized quadratic functions} remain 
some of the
\textit{the most difficult objectives} for gradient methods.

Moreover, the convergence rate from~\eqref{FuncResFinal}
was established for the composite version
of the Fast Gradient Method~\cite{roulet2017sharpness,nesterov_inexact_2022}.
Therefore, our result shows that this rate is tight for 
regularized quadratic objectives.
It remains an interesting open question, whether it is possible
to achieve the corresponding bound~\eqref{GradNormFinal} for minimizing the gradient norm.

\paragraph{Strongly Convex Functions} %

Now, let us demonstrate that Theorem~\ref{thm:main_thm}
also allows us to recover the classical lower bounds well-known in the literature.
We set $p = 2$ to analyze the class of strongly convex functions with Lipschitz continuous gradients.
A direct consequence of Theorem~\ref{thm:main_thm} is the following statement.

\begin{corollary}\label{cor:strongcvx}
    Let $r>0$, $s > 0$, $\mu > 0$, and $L > \mu$. For any first-order method that
    performs $1 \leq N \leq \frac{d - 1}{2}$ iterations, there is a strongly convex quadratic function $f$
    with $\lambda_{\min}(A) = \mu$, $\lambda_{\max}(A) = L$, and $\| \xopt \| = r$, such that the functional residual of the last iteration is lower bounded by
    $$
    \ba{rcl}
        f(x_N) - f_\star & \geq & 
        sr^2 \exp \Big(-\frac{16N}{\Qmod^{1/2} - 1}\Big),
    \ea
    $$
    and
    $$
    \ba{rcl}
    \| \nabla f(x_N) \| & \geq &
    sr \exp \Bigl( -\frac{8N}{\Qmod^{1/2} - 1} \Bigr),
    \ea
    $$
    where $r = \|\xopt\|$ and $\Qmod = \frac{L + s}{\mu + s}$ is the modified condition number of $f$.
\end{corollary}

For $s > 0$, this bound is stronger than the classic bound for smooth strongly convex 
functions as the numerator and denominator of our modified condition number 
have been shifted by the regularization parameter $s$. One should note that when $p=2$,
\begin{equation*}
    f(x) = \frac{1}{2} x^\top A x - b^\top x + \frac{s}{p} \|x\|^2 = \frac{1}{2} x^\top (A + sI) x - b^\top x,
\end{equation*}
and problems of the form in~\eqref{eq:prob} correspond directly to instances of quadratic functions. The theory we present thus also recovers existing complexity results from the literature on smooth strongly convex functions.

\paragraph{Convex Functions} %

Let us now consider a setting in which problems of the form in~\eqref{eq:prob} are equivalent to instances of the trust-region subproblem. For this section, we let $s = \frac{1}{D^p}$, where $D > 0$ is a fixed parameter, and consider what happens when $p \to +\infty$. Then the problem of the form in~\eqref{eq:prob} is equivalent to
\begin{equation*}
    \min_{\substack{x \in \R^d: \\ \|x\| \leq D}} \Bigg[f(x) = \frac{1}{2}x^\top A x - b^\top x \Bigg],
\end{equation*}
and becomes an instance of a constrained optimization problem. 
The corresponding bound from Theorem~\ref{thm:final} is
$$
\ba{rcl}
f(x_N) - f_{\star} & = &
\Omega\left( \Bigl[ \frac{LD^2}{N^2} \Bigr]^{\frac{p}{p - 2}}  \right)
\;\; \underset{p \to +\infty}{\longrightarrow} \;\;
\Omega\left( \frac{LD^2}{N^2} \right),
\ea
$$
which is the standard bound for minimizing 
smooth convex functions over the Euclidean ball.
At the same time, it is also known that 
the accelerated version of the projected gradient method \cite{nesterov_lectures_2018}
achieves the convergence rate of $\ub(\frac{L D^2}{N^2})$,
which matches the lower bound.
We conclude that our theory appears consistent with existing results on this class of trust-region subproblems.

\section{Experiments}\label{sec:experiments}

To validate our theoretical results, we construct the worst-case instances described throughout Section~\ref{sec:lower_bounds} and measure the performance of a Krylov subspace solver in terms of functional residual. In order to avoid problems with conditioning, our implementation of the Krylov subspace solver does not optimize in the subspace spanned by the vectors $b, \Ar b, \ldots, \Ar^{k-1} b$ directly, but rather over a basis generated for the $k$-th order Krylov subspace using the Lanczos process~\cite{chen2024lanczos}.

Our first experiment compares the performance of a Krylov subspace solver with the gradient method using the step size outlined in~\eqref{GradResBound2} on a single instance of our construction. We also examine the performance of these two methods on randomly generated instances of our problem (for details, see Appendix~\ref{app:random_qps}) in Figure~\ref{fig:single_runs}. The point of this experiment is to see what the difference in performance is between the two methods, and how our construction compares with more natural instances of our problem class. We found that the gradient method using a constant step size is suboptimal compared to our Krylov subspace solver, and that our construction is indeed far more difficult to minimize in terms of functional residual.

To test the validity of our general lower bound construction, we run an experiment where we measure the lowest functional residual measured on our construction as a function of the maximum number of iterations we run using a Krylov subspace solver, and the gradient method using the step size defined in~\eqref{GradResBound2}. Since our lower bound on our construction is directly dependent on the number of iterations $N$ we intend to run using a given first-order method, we construct individual problem instances for a variety of values between $1 \leq N \leq \frac{d-1}{2}$ and measure the functional residual after running the solver for $N$ iterations. Due to numerical difficulties associated with constructing a vector $\pi \in \Delta_{2N+1}$ according to the proof of Lemma~\ref{lem:min_max_sol}, we instead elect to use a heuristic for $\pi$. Details about the implementation of this experiment are available in Appendix~\ref{app:exp_imp} and details about the difficulties associated with constructing $\pi$ are available in Appendix~\ref{app:exp_imp_constr}. Figure~\ref{fig:verifying_lower_bound_and_pi_experiment} (left) contains the results of this first experiment. We also compare the construction generated by our heuristic estimation of the vector $\pi$ with the uniform distribution $\pi_i \coloneqq \frac{1}{2N+1}$. The results of this experiment can be found in Figure~\ref{fig:verifying_lower_bound_and_pi_experiment} (right). The theoretical framework describes the performance of our construction using our heuristic exceptionally well, and it appears to capture the trend for the uniformly distributed $\pi$ reasonably well until the last few instances if we ignore constants.

To further explore relaxations on our construction, we use this uniform distribution of $\pi$ and examine whether we can generate difficult problem instances using a random eigenspectrum of $A$ sampled from a beta distribution. We generate such a problem instance with the same problem parameters as in the previous experiment, except we set\footnote{We elected for a smaller-dimensional problem due to the problem of finding solutions to Krylov subspaces in higher dimensions being badly conditioned on the random problem instance.} $d=100$, and compare their functional residual as a function of the number of iterations we run with a problem instance where the eigenvalues correspond to the Chebyshev extrema discussed in Section~\ref{sec:lower_bounds}. The results of that experiment, and a visualization of the sampled eigenvalues, are available in Figure~\ref{fig:eigenval_experiment}.

\newpage

\begin{figure}[h!]
    \centering
    \subfigure[]{\includegraphics[width=0.48\textwidth]{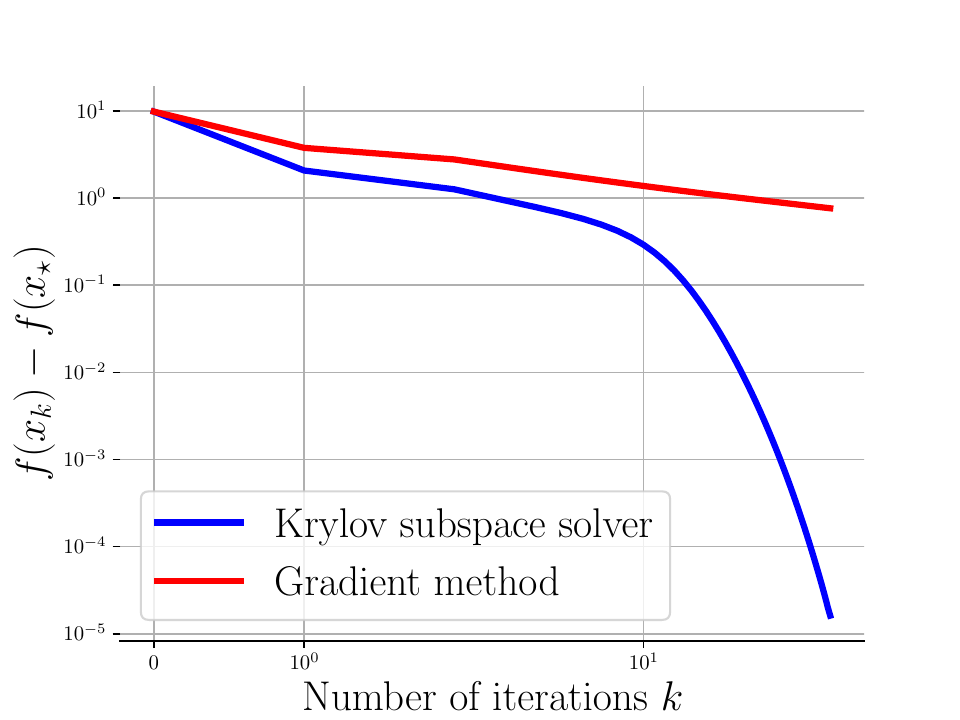}\label{fig:single_run_adversarial}}
    \subfigure[]{\includegraphics[width=0.48\textwidth]{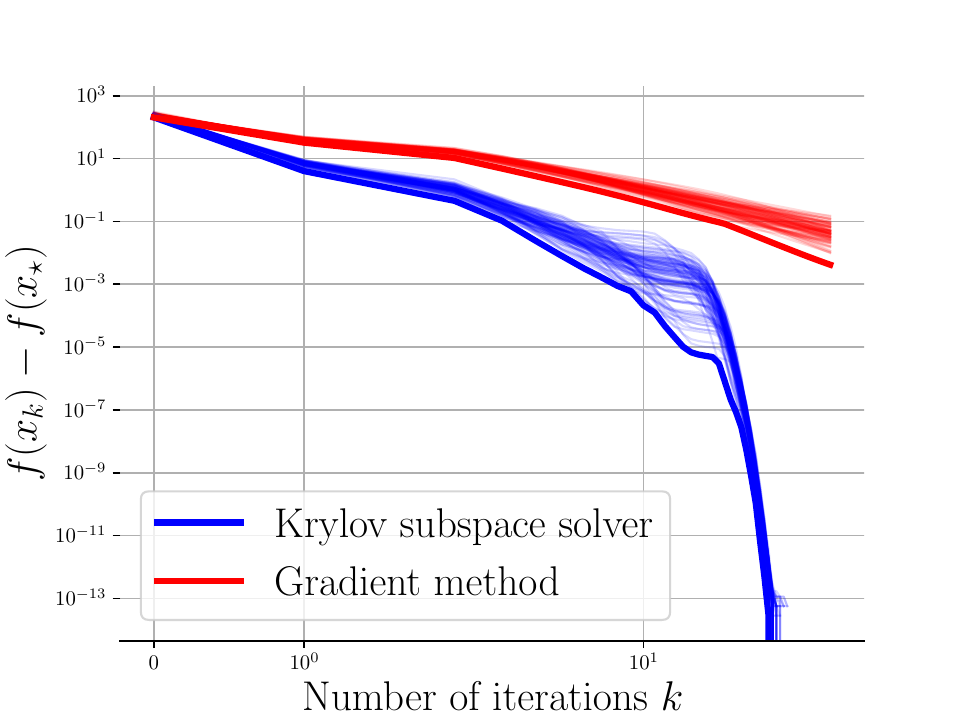}\label{fig:single_run_random}}
    \caption{The functional residual after $N$ steps of running our implementation of a Krylov subspace solver (blue), as compared to the gradient method for the same number of steps (red) on~\subref{fig:single_run_adversarial} adversarially constructed instances, and~\subref{fig:single_run_random} on randomly generated problem instances. In~\subref{fig:single_run_random} we also highlight the per-iteration lowest functional residual with a solid line to illustrate the worst-case performance across the random instances.
    }
    \label{fig:single_runs}
\end{figure}

\begin{figure}[h!]
    \centering
    \includegraphics[width=0.48\textwidth]{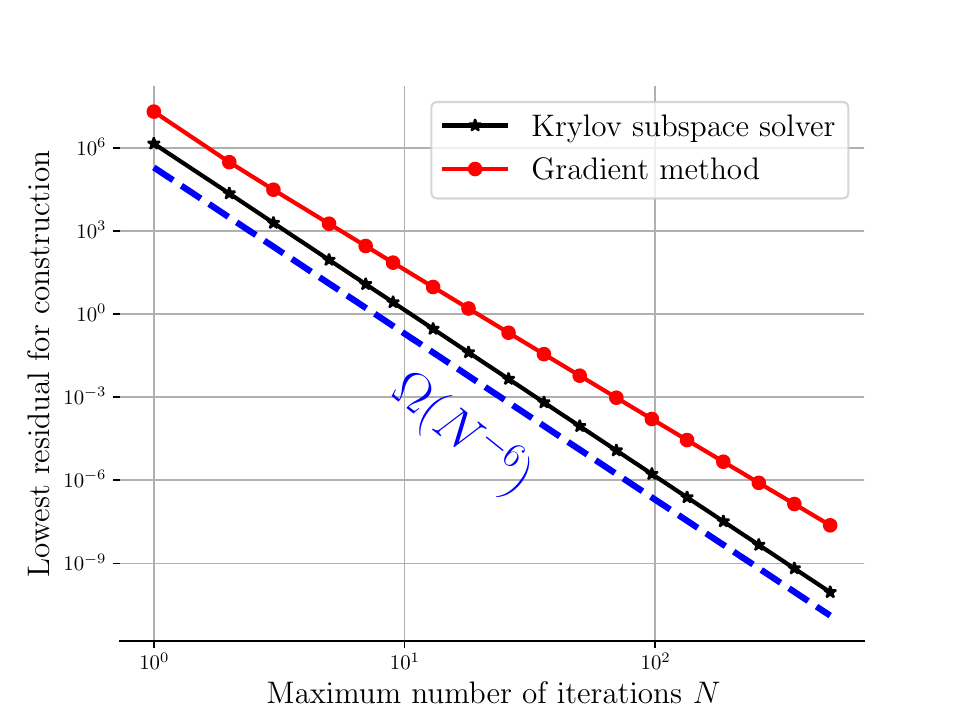}
    \includegraphics[width=0.48\textwidth]{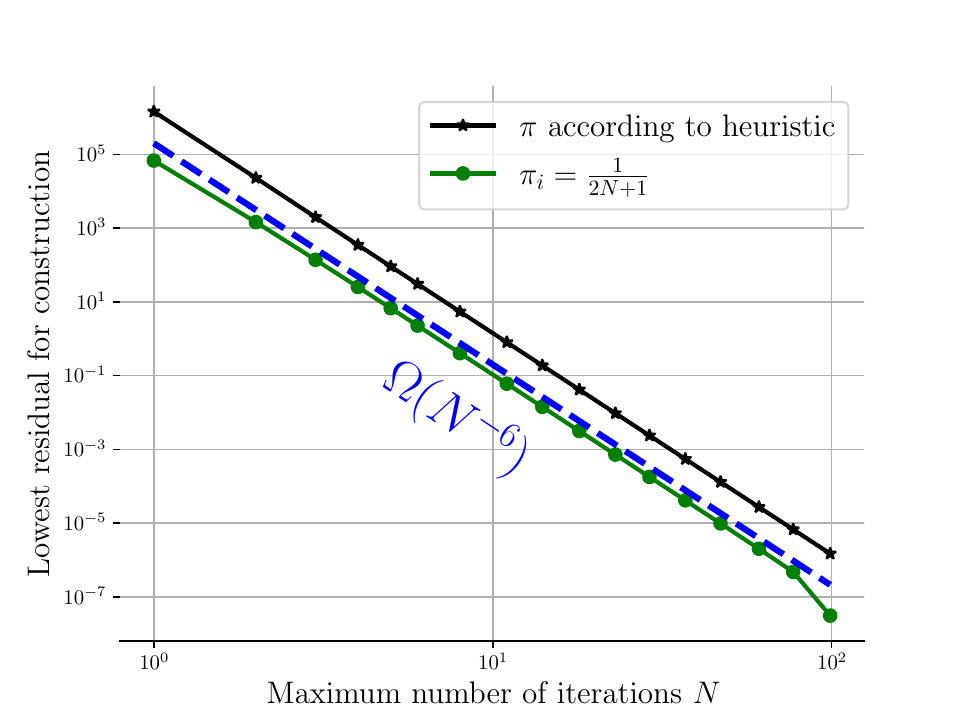}
    \caption{
    Result of running a Krylov subspace solver on our adversarially constructed problem instances with $\pi$ generated using our heuristic (black).
    On the left graph we compare its performance with the Gradient Method (red). On the right graph, we compare the performance with the problem instance with evenly allocating mass to all components of $\pi$ (red).
    Each point along the $x$-axis corresponds to a separate construction, and the value along the $y$-axis is the minimal functional residual achieved by the solver after $N$ iterations. 
    The blue dashed line corresponds to the asymptotic trend predicted by our theory when $p=3$. Our theory lower bounds the performance in both cases, but it is likely a tight estimate for the worst-case performance of Krylov subspace solvers given that the slope matches exactly what is predicted in theory.}
    \label{fig:verifying_lower_bound_and_pi_experiment}
\end{figure}

\newpage

\begin{figure}[h!]
    \centering
    \subfigure[]{\includegraphics[width=0.48\textwidth]{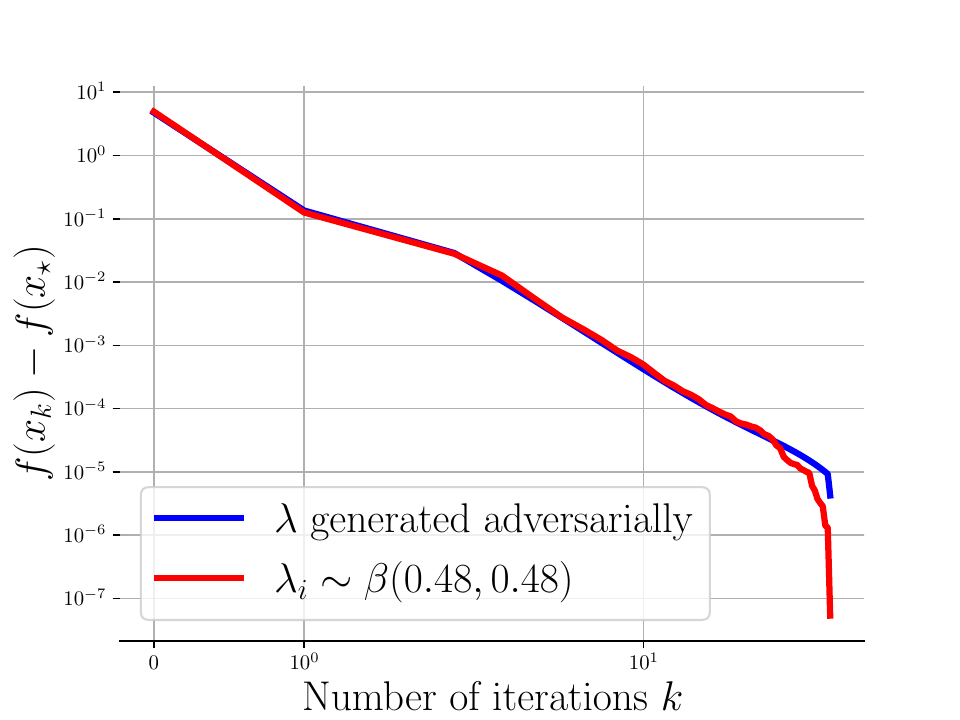}\label{fig:eigenval_experiment_graph}}
    \subfigure[]{\includegraphics[width=0.48\textwidth]{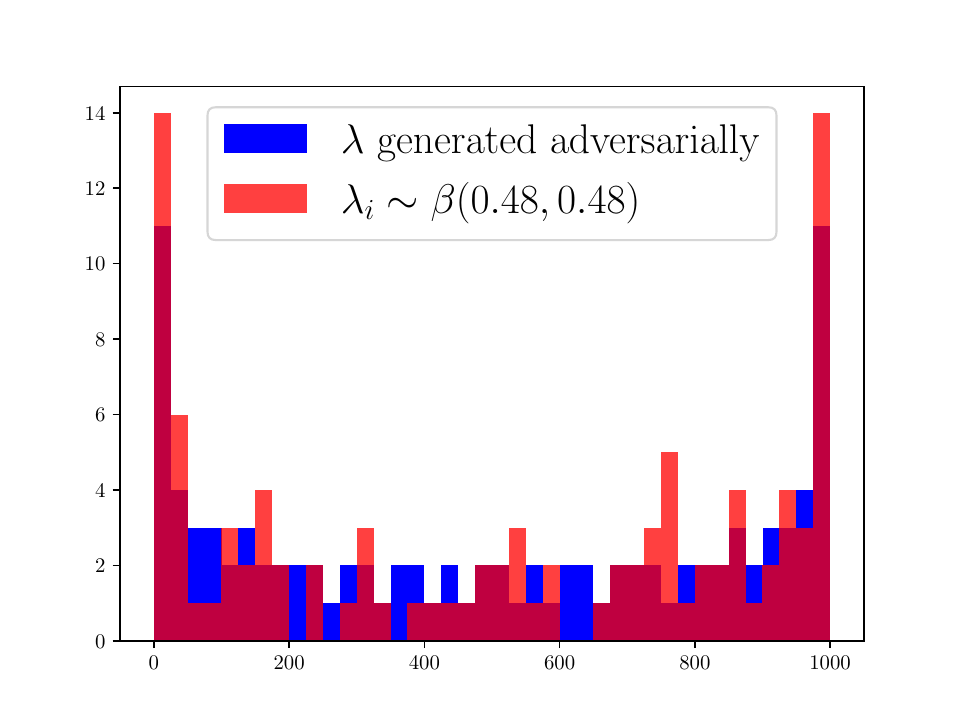}\label{fig:eigenval_experiment_hist}}
    \caption{Result of running a Krylov subspace solver on our adversarially constructed problem instances with eigenvalues of $A$ generated according to the construction in blue, and randomly according to a beta distribution in red. \subref{fig:eigenval_experiment_graph} contains the functional residual at every step of the method, and \subref{fig:eigenval_experiment_hist} is a histogram of the eigenvalues. The beta distribution appears to achieve a good approximation of the adversarially placed eigenvalues, as seen in both plots.}
    \label{fig:eigenval_experiment}
\end{figure}

\section{Discussion}

In this paper, we present the oracle complexity for gradient methods
minimizing the class of uniformly convex regularized quadratic functions. It contains a shortened version of a broader framework adapted from a set of lecture notes by Arkadi Nemirovski \cite{arkadi_nemirovski_information-based_1994} and developed in an unpublished master's thesis~\cite{master_thesis}. It covers both existing results for smooth and strongly convex functions and establishes new rates for the classes of uniformly convex functions. 

Our results confirm that quadratic functions constitute the main computational difficulty 
for the first-order methods. At the same time, exploiting the uniform convexity of the problem,
we see that the rate can be significantly improved compared to purely convex rates.
Thus, for the case $p = 3$ (cubically regularized convex quadratic functions), 
the rate of one-step methods such as the basic gradient descent is $\mathcal{O}(1 / N^3)$
and the optimal rate becomes $\mathcal{O}(1 / N^6)$, which is significantly better
than the standard rates of the order $\mathcal{O}(1 / N)$ and $\mathcal{O}(1 / N^2)$ correspondingly.

\bibliographystyle{plain}
\bibliography{ref}

\begin{thebibliography}{10}

\bibitem{agarwal_lower_2018}
Naman Agarwal and Elad Hazan.
\newblock Lower bounds for higher-order convex optimization.
\newblock In Sébastien Bubeck, Vianney Perchet, and Philippe Rigollet, editors, {\em Proceedings of the 31st Conference On Learning Theory}, volume~75 of {\em Proceedings of Machine Learning Research}, pages 774--792. PMLR, 06--09 Jul 2018.

\bibitem{agrawal2018rewriting}
Akshay Agrawal, Robin Verschueren, Steven Diamond, and Stephen Boyd.
\newblock A rewriting system for convex optimization problems.
\newblock {\em Journal of Control and Decision}, 5(1):42--60, 2018.

\bibitem{arjevani_oracle_2019}
Yossi Arjevani, Ohad Shamir, and Ron Shiff.
\newblock Oracle complexity of second-order methods for smooth convex optimization.
\newblock {\em Mathematical Programming}, 178(1):327--360, November 2019.

\bibitem{bauschke2017descent}
Heinz~H Bauschke, J{\'e}r{\^o}me Bolte, and Marc Teboulle.
\newblock A descent lemma beyond {L}ipschitz gradient continuity: first-order methods revisited and applications.
\newblock {\em Mathematics of Operations Research}, 42(2):330--348, 2017.

\bibitem{master_thesis}
Daniel Berg~Thomsen.
\newblock Oracle lower bounds for minimizing regularized quadratic functions.
\newblock Master's thesis, École Polytechnique Fédérale de Lausanne, January 2024.

\bibitem{carmon_analysis_2018}
Yair Carmon and John~C Duchi.
\newblock Analysis of {Krylov} {Subspace} {Solutions} of {Regularized} {Non}-{Convex} {Quadratic} {Problems}.
\newblock In {\em Advances in {Neural} {Information} {Processing} {Systems}}, volume~31. Curran Associates, Inc., 2018.

\bibitem{cartis_adaptive_2011-1}
Coralia Cartis, Nicholas I.~M. Gould, and Philippe~L. Toint.
\newblock Adaptive cubic regularisation methods for unconstrained optimization. {Part} {I}: motivation, convergence and numerical results.
\newblock {\em Mathematical Programming}, 127(2):245--295, April 2011.

\bibitem{chen2024lanczos}
Tyler Chen.
\newblock The {L}anczos algorithm for matrix functions: a handbook for scientists.
\newblock {\em arXiv preprint arXiv:2410.11090}, 2024.

\bibitem{cheney_introduction_1998}
E.W. Cheney.
\newblock {\em Introduction to {Approximation} {Theory}}.
\newblock {AMS} {Chelsea} {Publishing} {Series}. AMS Chelsea Pub., 1998.

\bibitem{chowdhury_iterative_2018}
Agniva Chowdhury, Jiasen Yang, and Petros Drineas.
\newblock An {Iterative}, {Sketching}-based {Framework} for {Ridge} {Regression}.
\newblock In {\em Proceedings of the 35th {International} {Conference} on {Machine} {Learning}}, pages 989--998. PMLR, July 2018.
\newblock ISSN: 2640-3498.

\bibitem{conn_trust_2000}
Andrew~R. Conn, Nicholas I.~M. Gould, and Philippe~L. Toint.
\newblock Trust {Region} {Methods} {\textbar} {SIAM} {Publications} {Library}, 2000.

\bibitem{diakonikolas2024complementary}
Jelena Diakonikolas and Crist{\'o}bal Guzm{\'a}n.
\newblock Complementary composite minimization, small gradients in general norms, and applications.
\newblock {\em Mathematical Programming}, pages 1--45, 2024.

\bibitem{diamond2016cvxpy}
Steven Diamond and Stephen Boyd.
\newblock {CVXPY}: {A} {P}ython-embedded modeling language for convex optimization.
\newblock {\em Journal of Machine Learning Research}, 17(83):1--5, 2016.

\bibitem{doikov_lower_2022}
Nikita Doikov.
\newblock Lower {Complexity} {Bounds} for {Minimizing} {Regularized} {Functions}, February 2022.
\newblock arXiv:2202.04545 [math].

\bibitem{doikov_minimizing_2021}
Nikita Doikov and Yurii Nesterov.
\newblock Minimizing {Uniformly} {Convex} {Functions} by {Cubic} {Regularization} of {Newton} {Method}.
\newblock {\em Journal of Optimization Theory and Applications}, 189(1):317--339, April 2021.

\bibitem{erway_subspace_2010}
Jennifer~B. Erway and Philip~E. Gill.
\newblock A {Subspace} {Minimization} {Method} for the {Trust}-{Region} {Step}.
\newblock {\em SIAM Journal on Optimization}, 20(3):1439--1461, January 2010.
\newblock Publisher: Society for Industrial and Applied Mathematics.

\bibitem{gorbunov2024methods}
Eduard Gorbunov, Nazarii Tupitsa, Sayantan Choudhury, Alen Aliev, Peter Richt{\'a}rik, Samuel Horv{\'a}th, and Martin Tak{\'a}{\v{c}}.
\newblock Methods for convex $(l\_0, l\_1) $-smooth optimization: Clipping, acceleration, and adaptivity.
\newblock {\em arXiv preprint arXiv:2409.14989}, 2024.

\bibitem{gould_solving_1999}
Nicholas I.~M. Gould, Stefano Lucidi, Massimo Roma, and Philippe~L. Toint.
\newblock Solving the {Trust}-{Region} {Subproblem} using the {Lanczos} {Method}.
\newblock {\em SIAM Journal on Optimization}, 9(2):504--525, January 1999.
\newblock Publisher: Society for Industrial and Applied Mathematics.

\bibitem{gould_solving_2010}
Nicholas I.~M. Gould, Daniel~P. Robinson, and H.~Sue Thorne.
\newblock On solving trust-region and other regularised subproblems in optimization.
\newblock {\em Mathematical Programming Computation}, 2(1):21--57, March 2010.

\bibitem{grapiglia_regularized_2017}
Geovani~N. Grapiglia and Yurii Nesterov.
\newblock Regularized {Newton} {Methods} for {Minimizing} {Functions} with {Hölder} {Continuous} {Hessians}.
\newblock {\em SIAM Journal on Optimization}, 27(1):478--506, January 2017.
\newblock Publisher: Society for Industrial and Applied Mathematics.

\bibitem{grapiglia_accelerated_2019}
Geovani~N. Grapiglia and Yurii Nesterov.
\newblock Accelerated {Regularized} {Newton} {Methods} for {Minimizing} {Composite} {Convex} {Functions}.
\newblock {\em SIAM Journal on Optimization}, 29(1):77--99, January 2019.
\newblock Publisher: Society for Industrial and Applied Mathematics.

\bibitem{griewank_modification_1981}
Andreas Griewank.
\newblock The modification of {Newton}'s method for unconstrained optimization by bounding cubic terms.
\newblock Technical report, 1981.

\bibitem{guzman_lower_2015}
Cristóbal Guzmán and Arkadi Nemirovski.
\newblock On lower complexity bounds for large-scale smooth convex optimization.
\newblock {\em Journal of Complexity}, 31(1):1--14, February 2015.

\bibitem{hastie_linear_2009}
Trevor Hastie, Robert Tibshirani, and Jerome Friedman.
\newblock Linear {Methods} for {Regression}.
\newblock In {\em The {Elements} of {Statistical} {Learning}: {Data} {Mining}, {Inference}, and {Prediction}}, Springer {Series} in {Statistics}, pages 43--99. Springer, New York, NY, 2009.

\bibitem{hazan_linear-time_2016}
Elad Hazan and Tomer Koren.
\newblock A linear-time algorithm for trust region problems.
\newblock {\em Mathematical Programming}, 158(1):363--381, July 2016.

\bibitem{hoffman_linear_1971}
K.~Hoffman and R.A. Kunze.
\newblock {\em Linear {Algebra}}.
\newblock Prentice-Hall, 1971.

\bibitem{kacham_sketching_2022}
Praneeth Kacham and David Woodruff.
\newblock Sketching {Algorithms} and {Lower} {Bounds} for {Ridge} {Regression}.
\newblock In {\em Proceedings of the 39th {International} {Conference} on {Machine} {Learning}}, pages 10539--10556. PMLR, June 2022.
\newblock ISSN: 2640-3498.

\bibitem{koloskova2023revisiting}
Anastasia Koloskova, Hadrien Hendrikx, and Sebastian~U Stich.
\newblock Revisiting gradient clipping: Stochastic bias and tight convergence guarantees.
\newblock In {\em International Conference on Machine Learning}, pages 17343--17363. PMLR, 2023.

\bibitem{lu2018relatively}
Haihao Lu, Robert~M Freund, and Yurii Nesterov.
\newblock Relatively smooth convex optimization by first-order methods, and applications.
\newblock {\em SIAM Journal on Optimization}, 28(1):333--354, 2018.

\bibitem{mason_chebyshev_2002}
J.~C. Mason and David~C. Handscomb.
\newblock {\em Chebyshev {Polynomials}}.
\newblock Chapman and Hall/CRC, New York, September 2002.

\bibitem{arkadi_nemirovski_information-based_1994}
Arkadi Nemirovski.
\newblock {\em Information-based complexity of convex programming}.
\newblock 1994.

\bibitem{nemirovskii1985optimal}
Arkadi Nemirovski and Yurii Nesterov.
\newblock Optimal methods of smooth convex minimization.
\newblock {\em USSR Computational Mathematics and Mathematical Physics}, 25(2):21--30, 1985.

\bibitem{nemirovski_problem_1983}
Arkadi Nemirovski and David Yudin.
\newblock Problem complexity and method efficiency in optimization.
\newblock 1983.
\newblock Publisher: Wiley-Interscience.

\bibitem{nesterov_method_1983}
Yurii Nesterov.
\newblock A method for solving the convex programming problem with convergence rate {O}(1/k{\textasciicircum}2).
\newblock {\em Proceedings of the USSR Academy of Sciences}, 1983.

\bibitem{nesterov_accelerating_2008}
Yurii Nesterov.
\newblock Accelerating the cubic regularization of {Newton}’s method on convex problems.
\newblock {\em Mathematical Programming}, 112(1):159--181, March 2008.

\bibitem{nesterov_lectures_2018}
Yurii Nesterov.
\newblock {\em Lectures on {Convex} {Optimization}}, volume 137 of {\em Springer {Optimization} and {Its} {Applications}}.
\newblock Springer International Publishing, Cham, 2018.

\bibitem{nesterov_implementable_2021}
Yurii Nesterov.
\newblock Implementable tensor methods in unconstrained convex optimization.
\newblock {\em Mathematical Programming}, 186(1):157--183, March 2021.

\bibitem{nesterov_inexact_2022}
Yurii Nesterov.
\newblock Inexact basic tensor methods for some classes of convex optimization problems.
\newblock {\em Optimization Methods and Software}, 37(3):878--906, May 2022.
\newblock Publisher: Taylor \& Francis.

\bibitem{nesterov2024primal}
Yurii Nesterov.
\newblock Primal subgradient methods with predefined step sizes.
\newblock {\em Journal of Optimization Theory and Applications}, pages 1--33, 2024.

\bibitem{nesterov_cubic_2006}
Yurii Nesterov and Boris~T. Polyak.
\newblock Cubic regularization of {Newton} method and its global performance.
\newblock {\em Mathematical Programming}, 108(1):177--205, August 2006.

\bibitem{nocedal2006conjugate}
Jorge Nocedal and Stephen~J Wright.
\newblock Conjugate gradient methods.
\newblock {\em Numerical optimization}, pages 101--134, 2006.

\bibitem{polyak1987introduction}
Boris~T Polyak.
\newblock Introduction to optimization.
\newblock 1987.

\bibitem{roulet2017sharpness}
Vincent Roulet and Alexandre d'Aspremont.
\newblock Sharpness, restart and acceleration.
\newblock {\em Advances in Neural Information Processing Systems}, 30, 2017.

\bibitem{van2017forward}
Quang Van~Nguyen.
\newblock Forward-backward splitting with {B}regman distances.
\newblock {\em Vietnam Journal of Mathematics}, 45(3):519--539, 2017.

\bibitem{vankov2024optimizing}
Daniil Vankov, Anton Rodomanov, Angelia Nedich, Lalitha Sankar, and Sebastian~U Stich.
\newblock Optimizing $(l\_0, l\_1) $-smooth functions by gradient methods.
\newblock {\em arXiv preprint arXiv:2410.10800}, 2024.

\bibitem{zhang2019gradient}
Jingzhao Zhang, Tianxing He, Suvrit Sra, and Ali Jadbabaie.
\newblock Why gradient clipping accelerates training: A theoretical justification for adaptivity.
\newblock {\em arXiv preprint arXiv:1905.11881}, 2019.

\end{thebibliography}

\newpage

\appendix
\section{Additional experiments} 

Figure~\ref{fig:functional_residual_grid} contains a grid of plots of the functional residual for the gradient method and a Krylov subspace solver when varying the problem parameters $L$ and $s$ of randomly generated problem instances.

\begin{figure}[ht]
    \centering
    \includegraphics[width=\textwidth]{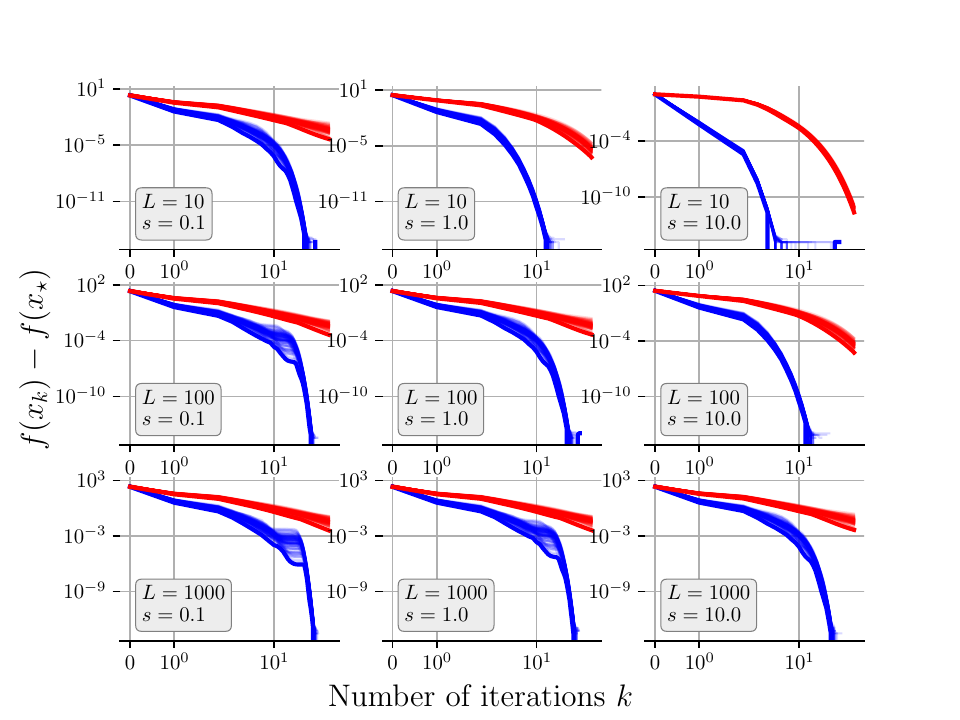}
    \caption{Grid of functional residuals resulting from running the gradient method using the step size defined in~\eqref{GradResBound2}, and a Krylov subspace solver on randomly generated problem instances. Different rows correspond to a fixed setting of $L$, and columns to a fixed setting of $s$. The transparent lines are individual problem instances, and the per-iteration lowest functional residual has been highlighted using an opaque line. Note that the initial functional residual is highly dependent on the setting of $L$, and that larger settings of $s$ lead to faster convergence for both methods.}
    \label{fig:functional_residual_grid}
\end{figure}

\subsection{Performance of One-Step Methods}

In Figure~\ref{fig:experiment_upper_bound} we measure the functional residual of the gradient method with the step size described in~\eqref{GradResBound2} and compare it with the upper bound described in Theorem~\ref{TheoremGMRate}. Our upper bound is indeed valid in this instance, but with a suboptimality gap that is not bounded by a constant. We observe the same for other parameter settings when generating the problem instances.
\begin{figure}[ht]
    \centering
    \includegraphics[width=0.43 \textwidth]{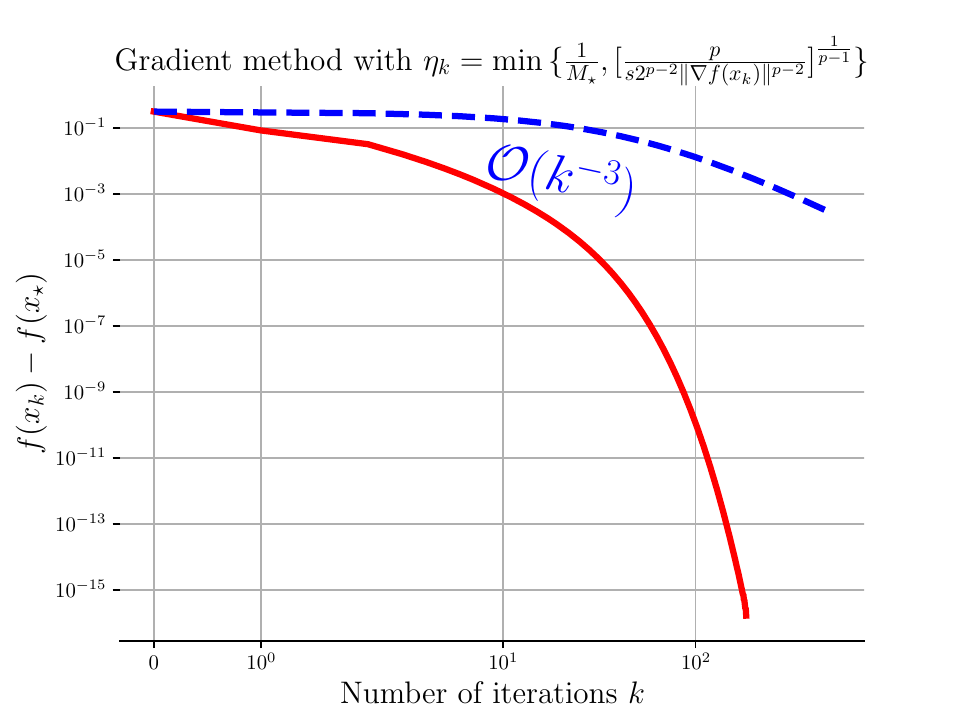}
    \caption{Functional residual of the gradient method with our step size. The problem instance has been randomly generated (according to Appendix~\ref{app:random_qps}), with $d = 1000$, $\mu = 0$, $L = $, $p = 3$, $s = 0.1$, and $\|\xopt\| = 1$. We see that in practice the method can perform better than the corresponding lower bound.
    }
    \label{fig:experiment_upper_bound}
\end{figure}

In Figure~\ref{fig:experiment_one_step} we compare the gradient method and the composite gradient method on our construction from Section~\ref{sec:one_step_lower_bounds}.
\begin{figure}[ht]
    \centering
    \includegraphics[width=0.43 \textwidth]{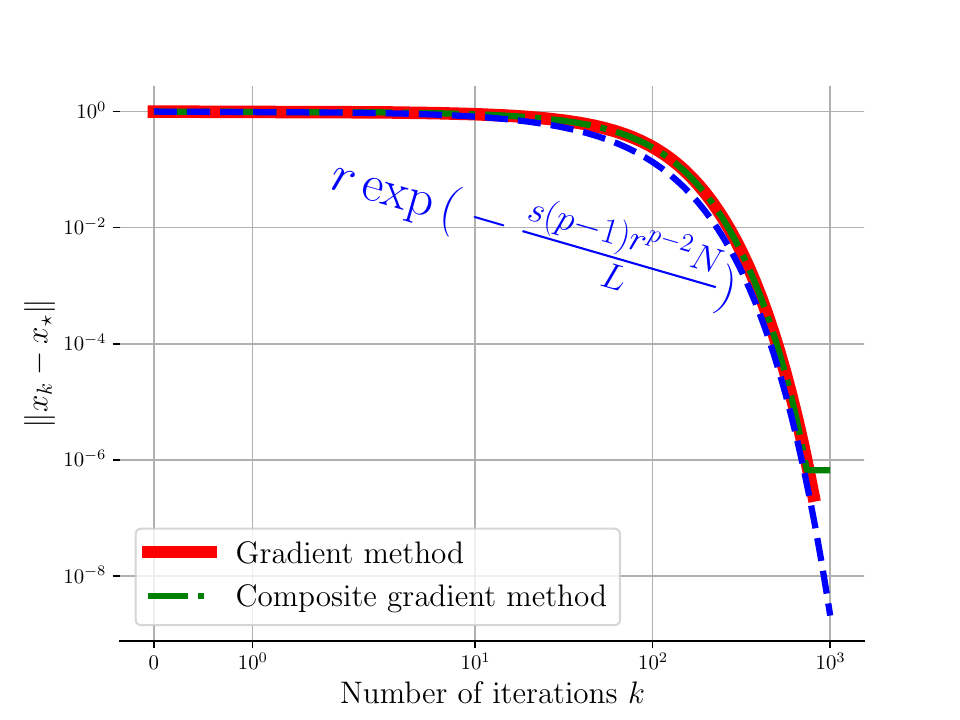}
    \caption{Performance of the gradient method and composite gradient method on our construction for one-step methods. We use the following step size for the gradient method, $\eta = \frac{1}{L+sr^{p-2}}$, and $\eta = \frac{1}{L}$ for the composite gradient method. The blue dashed line corresponds to our lower bound. The lower bound is below the lines of both methods as described by our theory, and the optimality gap increases with the number of iterations. We can also see that the two methods perform the same in this setting. We set the number of dimensions to 100, $\mu = 0$, $L = 100$, $p = 3$, $s = 1$, and $\|x_0 - \xopt\| = \|x_0\| = 1$.}
    \label{fig:experiment_one_step}
\end{figure}

\section{Experimental details}\label{app:exp_details}

This section describes a set of details pertinent to implementing our experiments. All source code used to generate the figures in the paper is open source and available on GitHub\footnote{\href{https://github.com/danielbergthomsen/regularized-quadratics}{https://github.com/danielbergthomsen/regularized-quadratics}}.
\subsection{Experiment implementation}\label{app:exp_imp}
In order to implement a Krylov subspace solver, we can solve for a vector $c \in \R^{2k+1}$ such that
\beq \label{eq:krylov_step}
c^* \in \argmin_{c \in \R^{2k+1}}  f(\mathcal{K}_{2k+1} c),
\eeq
where $\mathcal{K}_i$ is the Krylov matrix whose columns consist of $b, Ab, A^2 b, \ldots, A^{i-1} b$, and $k$ is the current iteration. The iteration chosen by the solver is then given by
$$
x_k = \mathcal{K}_{2k+1} c.
$$
The matrix $\mathcal{K}_i$ becomes increasingly ill-conditioned as more columns are added, and directly optimizing using this basis leads to inaccurate solutions. The matrix may be replaced by a matrix whose columns consist of the basis generated by the Lanczos process instead. 

To find minimizers for a given regularized quadratic, we use a Newton-type scheme that finds a root of the equation
\beq
h(r) = r - \|(A + sr^{p-2} I)^{-1}b\|,
\eeq
where $r$ acts as an approximation to the norm of the solution $||\xopt||$. This routine is performed at each step to compute~\eqref{eq:krylov_step}. Details about this can be found in~\cite{conn_trust_2000, cartis_adaptive_2011-1}.

\subsection{Implementation details for our construction} \label{app:exp_imp_constr}
Let us fix a given choice of smallest and largest eigenvalues $\mu$ and $L$ respectively, and the number of iterations $N$ we are running our first-order method for. We need to construct the matrix $A$ and $b$ in accordance with our theory in order for our theoretical framework to apply. The first thing to note is that since the methods we are using in our experiments always choose iterations in the $2N$-th Krylov subspace, we will not need to adversarially construct orthogonal transformations as in the proof of Lemma~\ref{lem:kryltraj}. We can therefore set $A = \Diag(\lambda)$, where we will need to adversarially construct the eigenvalues $\lambda_1 \leq \lambda_2 \leq \ldots \lambda_{2N}$. Furthermore, since $b = A \xopt$, we can construct $b$ by first constructing $\xopt$, which is given by~\eqref{Xopt}. In other words, to construct $b$ we will need to specify the vector $\pi \in \Delta_{2N+1}$ corresponding to Lemma~\ref{lem:min_max_sol}. From the proof of Lemma~\ref{lem:min_max_sol} (see~\cite{carmon_analysis_2018}), the eigenvalues of the matrix $\Ar$ should correspond to extrema of the Chebyshev polynomial
$$
T_{2N}\Big(\frac{\Qmod + 1 - 2x/\mur}{\Qmod - 1}\Big),
$$
where $\Lr = L + sr^{p-2}$, $\mur = \mu + sr^{p-2}$, and $\Qmod = \Lr / \mur$. The eigenvalues of $A$ are thus given by
$$
\ba{rcl}
\lambda_k & = & \frac{1}{2} \Bigg[\Lr + \mur - (\Lr - \mur)\cos\Big(\frac{k\pi}{2N-1}\Big)\Bigg].
\ea
$$
The vector $\pi$ we need to construct $b$ is the solution to the max-min problem
$$
\max_{\pi \in \Delta_{2N+1}} \min_{q \in \mathcal{P}_{2N}} \sum_{k=1}^{2N+1} \pi_k \big[1 - x_k q(x_k)\big]^2.
$$
Using linear algebra notation, we may reformulate this into an equivalent problem as
\beq \label{eq:linear_alg_reform}
 \max_{\pi \in \Delta_{2N+1}} \underbrace{\min_{c \in \R^{2N}} \frac{1}{2}\|\sqrt{\Pi} \big[\mathcal{V}c - \mathbf{1}\big]\|^2}_{g(\pi)},
\eeq
where $\Pi = \Diag(\pi_i)$, and $\mathcal{V}$ is given by
$$
\ba{rcl}
\mathcal{V} &=& \begin{pmatrix}
            \lambda_1 & \lambda_1^2 & \cdots & \lambda_1^{2N} \\
            \lambda_2 & \lambda_2^2 & \cdots & \lambda_2^{d-1} \\
            \vdots & \vdots & \ddots & \vdots \\
            \lambda_{2N+1} & \lambda_{2N+1}^2 & \cdots & \lambda_{2N+1}^{2N}
        \end{pmatrix}.
\ea
$$
For a fixed vector $\pi$ of all non-zero coefficients, this corresponds to an instance of a least squares estimation problem with a 1-rank deficiency. As a consequence, the optimal vector $c \in \R^{2N}$ is given by
\beq \label{eq:c_opt}
\ba{rcl}
c^* &=& (\mathcal{V}^\top \Pi \mathcal{V})^{-1} \mathcal{V}^\top \pi.
\ea
\eeq
Plugging this back into~\eqref{eq:linear_alg_reform}, we end up with
\beq \label{eq:linear_alg_reform2}
\ba{rcl}
g(\pi) &=& \min_{c \in \R^{2N}} \frac{1}{2}\|\sqrt{\Pi}\big[\mathcal{V}c - \mathbf{1}\big] \|^2 \\
\\
&\equiv & \min_{q \in \R^{2N}} \frac{1}{2} (\mathcal{V}c)^\top \Pi \mathcal{V}c - \mathbf{1}^\top \Pi \mathcal{V}c \\
\\
&\stackrel{\eqref{eq:c_opt}}{=}& -\frac{1}{2} \pi^\top \mathcal{V} (\mathcal{V}^\top \Pi \mathcal{V})^{-1} \mathcal{V}^\top \pi +  \frac{1}{2}
\ea
\eeq
The problem is therefore equivalent to the convex optimization problem
\beq \label{eq:prob_cvx}
\ba{rcl}
\min_{A \succ 0, y} &\frac{1}{2} y^\top A^{-1} y \\
\\
\text{subject to: } &y = \mathcal{V}^\top \pi \\
 & A = \mathcal{V}^\top \Pi \mathcal{V}, \\
 & \pi \in \Delta_{2N+1},
\ea
\eeq
and solutions can be computed using a convex optimization solver.

The problem with the above-mentioned approach for computing $\pi$ is that~\eqref{eq:prob_cvx} is ill-conditioned, and directly computing high-dimensional solutions to sufficient accuracy is not feasible. In order to get around this limitation, we generate the solutions to problems of low dimensionality, fit a model to approximate $\pi$, and extrapolate to higher dimensions. See Figure~\ref{fig:optimal_pi} for examples of approximate solutions generated by a convex optimization solver.

\begin{figure}[ht]
    \centering
    \subfigure[]{\includegraphics[width=0.325\textwidth]{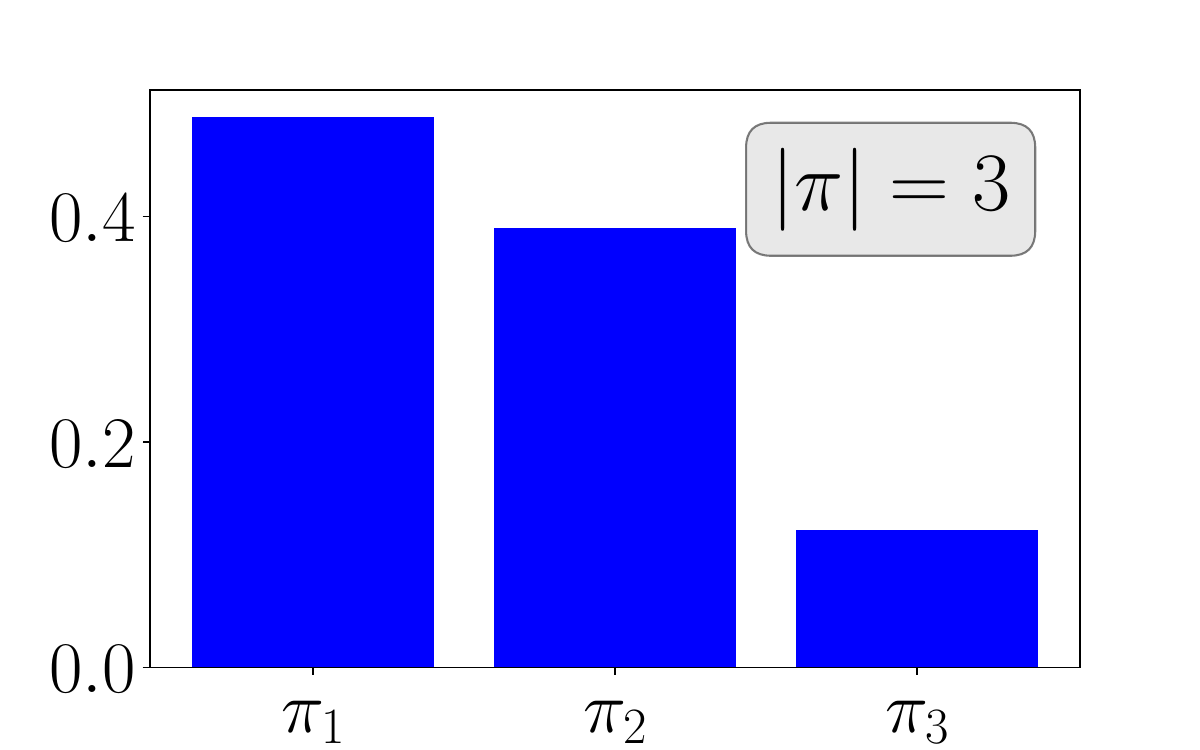}\label{fig:optimal_pi_3}}
    \subfigure[]{\includegraphics[width=0.325\textwidth]{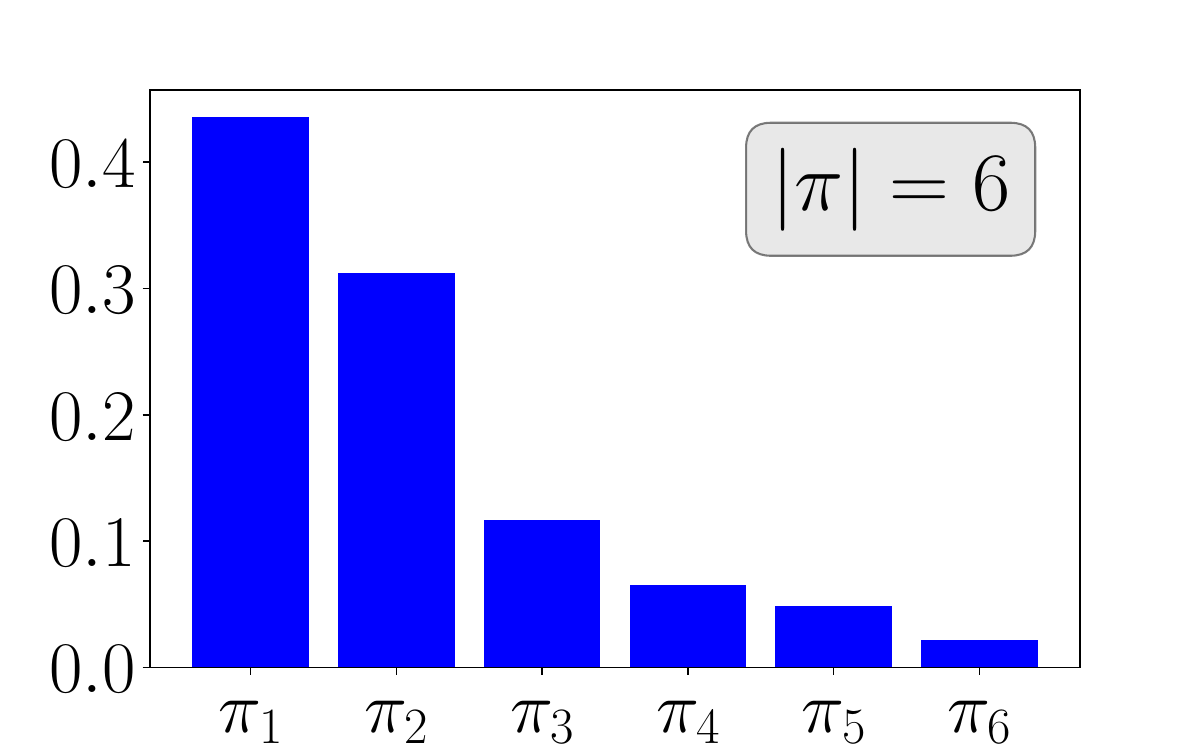}\label{fig:optimal_pi_6}}
    \subfigure[]{\includegraphics[width=0.325\textwidth]{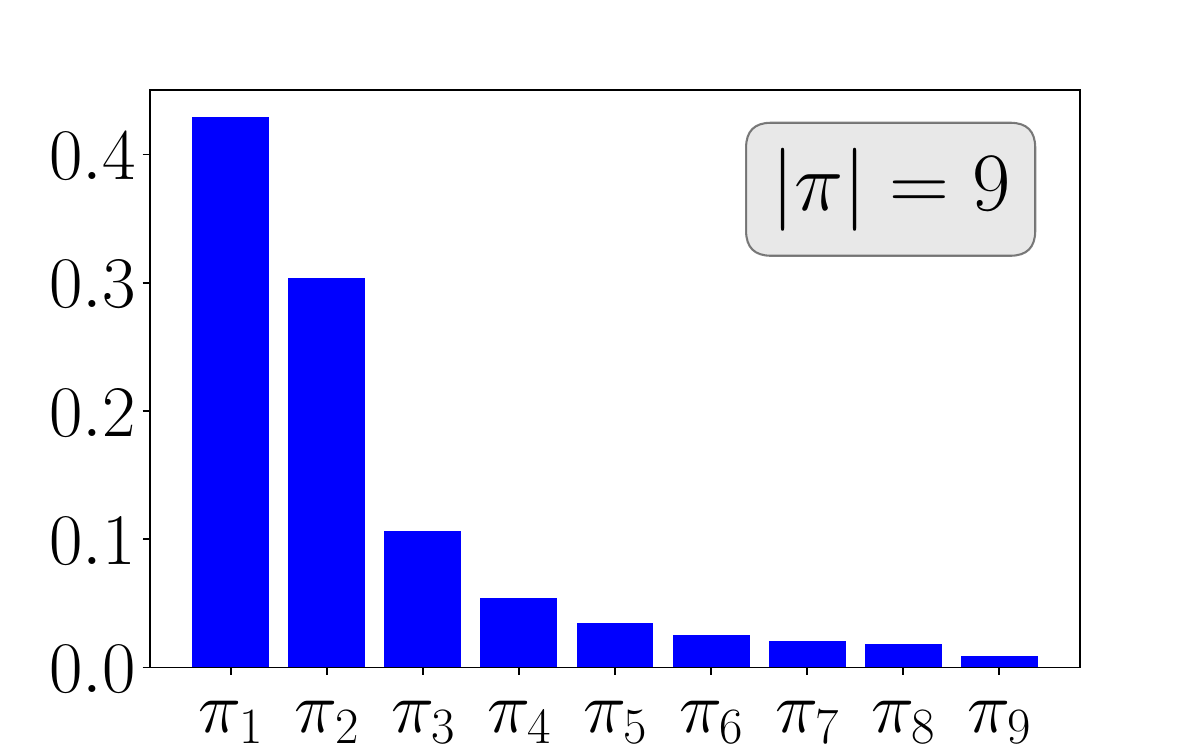}\label{fig:optimal_pi_9}}
    \caption{A few examples of the solution to the problem in~\eqref{eq:prob_cvx} found using the python package \texttt{cvxpy}~\cite{diamond2016cvxpy, agrawal2018rewriting} for different dimensions. The coefficients appear to decay exponentially, but inspecting~\subref{fig:optimal_pi_9} one can see that the first and last coefficients deviate slightly from the rate of decay of the rest of the coefficients.}
    \label{fig:optimal_pi}
\end{figure}

The approach we take to approximate $\pi$ is by fitting separate models to the first and last component as a function of the total number of components in $\pi$, and a single, separate model for all the components in between. The models used for the first and last components are the following:
$$
\hat{\pi_i} = c_1^{-c_2(i - c_3)} + c_4,
$$
where $i \in \{1, 2N+1\}$, and $c_1, c_2, c_3, c_4$ are fitted according to the best fit. If we denote by $u$ the vector containing the indices of the coefficients between the first and the last, then the model used to find the best estimate for these is given by
$$
\hat{\pi_i} = \Big[c_2 (u - c_1)\Big]^{-c_3}.
$$
The aggregated model approximates $\pi$ reasonably well in lower dimensions, and an example of its fit in comparison to the solution found by a convex optimization solver is found in Figure~\ref{fig:pi_hypothesis}.

\begin{figure}[ht]
    \centering
    \includegraphics[width=0.50\textwidth]{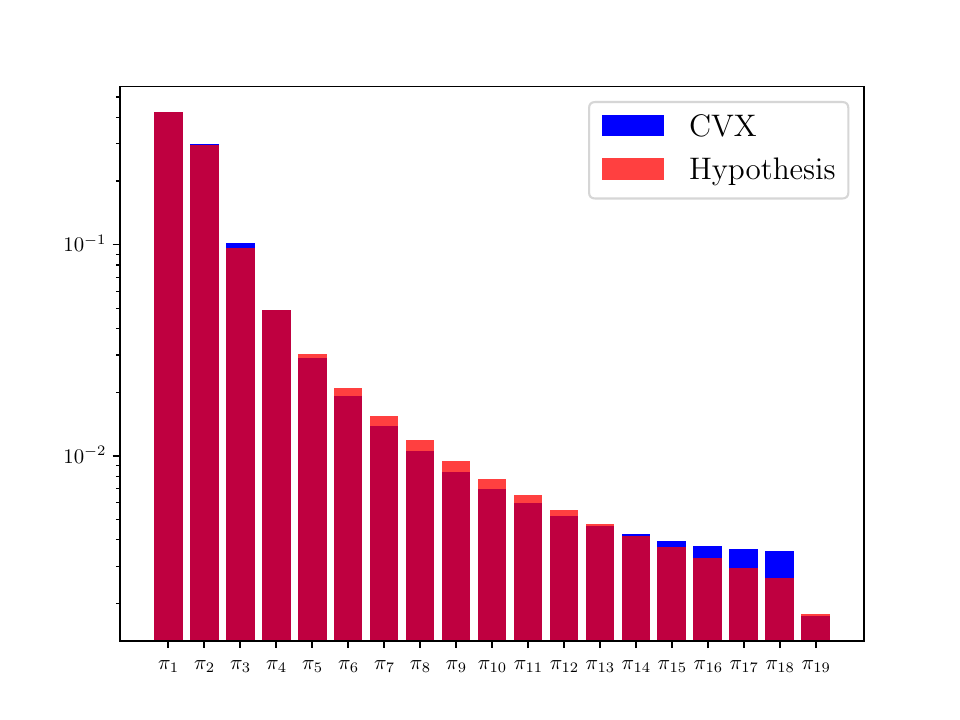}
    \caption{Histogram visualization of the coefficients generated by our hypothesis function, compared to the ground truth for $|\pi| = 19$.}
    \label{fig:pi_hypothesis}
\end{figure}

\subsection{Implementation details for randomly generated problem instances} \label{app:random_qps}
In order to generate random instances of~\eqref{eq:prob} while still being able to control dimension $d$ and the parameters $\mu, L, s, p$ and $\|\xopt\|$, we undertake the following procedure:
\begin{enumerate}
    \item Begin generating the eigenspectrum $\lambda$ by evenly dividing the interval $[\mu, L]$ into $d$ different points, with $\mu$ and $L$ being the smallest and largest eigenvalues respectively. Let $\Lambda = \Diag(\lambda_1, \ldots, \lambda_d)$.
    \item Generate $A$ by first sampling an orthogonal matrix $U$ according to the Haar distribution, which is a uniform distribution on the orthogonal group. Then set $A = U^\top \Lambda U$.
    \item Now generate $\xopt$ by first sampling a vector following a multivariate standard Gaussian distribution: $v \sim \mathcal{N}(0, 1)$. This vector will be uniformly distributed on some ball of an unknown radius. We can then set $\xopt = r \cdot \frac{v}{\|v\|}$ to control the radius of this ball.
    \item $b$ is now given by a closed formula as $b = \Ar \xopt$.
\end{enumerate}

\end{document}